\documentclass[a4,11pt]{amsart}
\usepackage{amssymb,amsmath,amsthm,txfonts}
\usepackage{cite}

\newtheorem{thm}{Theorem}[section]
\newtheorem{lem}[thm]{Lemma}

\newtheorem{prop}[thm]{Proposition}
\theoremstyle{definition}

\theoremstyle{remark}
\newtheorem{rem}[thm]{\textbf{Remark}}
\newtheorem{rems}[thm]{\textbf{Remarks}}

 \makeatletter
    
    \@addtoreset{equation}{section}
  \makeatother

\makeatletter
      \def\@makefnmark{%
         \leavevmode
            \raise.9ex\hbox{\check@mathfonts
                \fontsize\sf@size\z@\normalfont%
                            \@thefnmark}%
       }
      \makeatother

\newcommand{\D}{\textrm{div}}

\newcommand{\dd}{\textrm{d}}

\begin{document}

\title[]{The Navier-Stokes equations with the Neumann boundary condition in an infinite cylinder}
\author[]{K. Abe}
\date{}
\address[K. Abe]{Department of Mathematics, Graduate School of Science, Osaka City University, 3-3-138 Sugimoto, Sumiyoshi-ku Osaka, 558-8585, Japan}
\email{kabe@sci.osaka-cu.ac.jp}

\date{\today}

\maketitle


\begin{abstract}
We prove unique existence of local-in-time smooth solutions of the Navier-Stokes equations for initial data in $L^{p}$ and $p\in [3,\infty)$ in an infinite cylinder, subject to the Neumann boundary condition. 
\end{abstract}

\vspace{15pt}

\section{Introduction}
\vspace{10pt}
We consider the three-dimensional Navier-Stokes equations subject to the Neumann boundary condition:

\begin{equation*}
\begin{aligned}
\partial_t u-\Delta{u}+u\cdot \nabla u+\nabla{p}= 0,  \quad \D\ u&=0 \qquad \textrm{in}\ \Pi\times(0,T),  \\
\nabla\times  u\times n=0,\quad u\cdot n&=0 \qquad \textrm{on}\ \partial\Pi\times(0,T), \\
u&=u_0\hspace{18pt} \textrm{on}\ \Pi\times \{t=0\},
\end{aligned}
\tag{1.1}
\end{equation*}\\
for the infinite cylinder 

\begin{align*}
\Pi=\{ x=(x_1,x_2,x_3)\in \mathbb{R}^{3}\ |\ x_h=(x_1,x_2),\ |x_h|<1\ \}.
\end{align*}\\
Here, $n$ denotes the unit outward normal vector field on $\partial\Pi$. The local well-posedness of the Neumann problem (1.1) is established in \cite{FLR77}, \cite{Miyakawa80} for initial data in $L^{p}$, when $\Pi$ is smoothly bounded. See also \cite{Miyakawa81}, \cite{GM} for the Dirichlet problem. The purpose of this paper is to develop $L^{p}$-theory of (1.1) for the infinite cylinder $\Pi$. Let $L^{p}_{\sigma}$ denote the $L^{p}$-closure of $C_{c,\sigma}^{\infty}$, the space of all smooth solenoidal vector fields with compact support in $\Pi$. The main result of this paper is the following:  

\vspace{15pt}

\begin{thm}
For $u_0\in L^{p}_{\sigma}(\Pi)$ and $p\in [3,\infty)$, there exists $T>0$ and a unique solution $u\in C([0,T]; L^{p})\cap C^{\infty}(\overline{\Pi}\times (0,T])$ of (1.1) with the associated pressure $p\in C^{\infty}(\overline{\Pi}\times (0,T])$.
\end{thm}

\vspace{15pt}
The Neumann problem plays an important role in the theory of weak solutions to the Euler equations. When $\Pi$ is a two-dimensional bounded and simply-connected domain, global weak solutions to the Euler equations are constructed in \cite{Bardos72}, \cite{VK81}, \cite{MiyakawaYamada92} by taking a vanishing viscosity limit to (1.1). Since vorticity satisfies the homogeneous Dirichlet boundary condition subject to the Neumann boundary condition (1.1), $L^{p}$-norms of the vorticity are  uniformly bounded independently of viscosity. 

For the three-dimensional Cauchy problem, vanishing viscosity methods are applied in \cite{Swann71}, \cite{Kato72} to construct unique local-in-time solutions to the Euler equations in $\mathbb{R}^{3}$. It is unknown whether a vanishing viscosity method is applicable for domains with boundary. See \cite{EM70}, \cite{BB74}, \cite{Te75}, \cite{KatoLai} for local well-posedness results of the Euler equations. In \cite{A7}, the author studied vanishing viscosity limits of (1.1) for axisymmetric data based on the main result of this paper.

\vspace{15pt}

We outline the proof of Theorem 1.1. We extend the approach for bounded domains \cite{Miyakawa80}. We set the Laplace operator subject to the Neumann boundary condition 

\begin{equation*}
\begin{aligned}
&B u=-\Delta u\quad u\in D(B),\\
&D(B)=\{u\in W^{2,p}(\Pi)\ |\ \nabla\times u\times n=0,\ u\cdot n=0\quad \partial\Pi\  \}.
\end{aligned}
\tag{1.2}
\end{equation*}\\
When $\Pi$ is smoothly bounded, it is known that the operator $-B$ generates a $C_0$-analytic semigroup on $L^{p}$ for $p\in (1,\infty)$ \cite{Miyakawa80}, \cite{AKST}. We show analyticity of the semigroup for the infinite cylinder by using a solution formula for the resolvent problem. We then define a fractional power $B_0^{1/2}$ for the operator $B_0=B+\lambda_0$ and $\lambda_0>0$. Since the operator $B_0$ admits a bounded imaginary power \cite{Seeley71}, \cite{GHT}, the domain of the fractional power $D(B_0^{1/2})$ is continuously embedded to the Sobolev space $W^{1,p}$. 

We then define the Stokes operator as a restriction of the Laplace operator 

\begin{equation*}
\begin{aligned}
&A u= B u\quad u\in D(A),\\
&D(A)=D(B)\cap L^{p}_{\sigma}.
\end{aligned}
\tag{1.3}
\end{equation*}\\
Since the Laplace operator $B$ is commutable with the Helmholtz projection operator $\mathbb{P}$, the Stokes operator acts as an operator on the solenoidal vector space $L^{p}_{\sigma}$. By using the analyticity of the semigroup and boundedness of the Helmholtz projection operator on $L^{p}$ \cite{ST98}, we construct mild solutions $u\in C([0,T]; L^{p})$ for $u_0\in L^{p}_{\sigma}$ and $p\in[3,\infty)$ of the form

\begin{align*}
u=e^{-tA}u_0-\int_{0}^{t}e^{-(t-s)A}(\mathbb{P}u\cdot \nabla u)(s)\dd s. \tag{1.4}
\end{align*}\\

Higher regularity of mild solutions follow from elliptic estimates for the Helmholtz projection and the Stokes opeartor. We show that all derivatives of solutions belong to the H\"older space $C^{\mu}((0,T]; L^{s})$ for $\mu\in (0,1/2]$ and $s\in(3,\infty)$.\\

This paper is organized as follows. In Section 2, we show that the Laplace operator generates an analytic semigroup on $L^{p}$ for the infinite cylinder. In Section 3, we define a fractional power of the Laplace operator and prove a continuous embedding of the domain of the fractional power. In Section 4, we define the Stokes operator. In Section 5, we construct mild solutions. In Section 6, we prove higher regularity of mild solutions. In Appendix A, we prove higher regularity estimates for the Laplace operator in the infinite cylinder, used in Section 6. In Appendix B, we estimate the resolvent of the Laplace operator by a multiplier theorem.

\vspace{20pt}

\section{Resolvent estimates for the Laplace operator}
\vspace{10pt}

We start with a resolvent estimate for the Laplace operator subject to the Neumann boundary condition in the infinite cylinder. We derive a solution formula for the Neumann problem by using the resolvent of two-dimensional problems.

\vspace{15pt}

\subsection{A solution formula}

We consider the resolvent problem

\begin{equation*}
\begin{aligned}
\lambda u-\Delta u&=f\quad \textrm{in}\ \Pi,\\
\nabla\times u\times n=0,\ u\cdot n&=0\quad \textrm{on}\ \partial\Pi,
\end{aligned}
\tag{2.1}
\end{equation*}\\
for $\lambda\in \Sigma_{\theta}=\{\lambda\in \mathbb{C}\backslash \{0\}\ |\ |\textrm{arg}\lambda|<\theta  \}$ and $\theta\in (\pi/2,\pi)$. We use a solution formula and estimate the resolvent of (2.1). We consider two-dimensional problems in a unit disk $D=\{x\in \mathbb{R}^{2}\ |\ |x|<1 \}$:

\begin{equation*}
\begin{aligned}
\mu v-\Delta v&=g\quad \textrm{in}\ D,\\
\nabla^{\perp}\cdot v=0,\ v\cdot n&=0\quad \textrm{on}\ \partial D,
\end{aligned}
\tag{2.2}
\end{equation*}
\begin{equation*}
\begin{aligned}
\mu w-\Delta w&=h\quad \textrm{in}\ D,\\
\partial_n w&=0\quad \textrm{on}\ \partial D,
\end{aligned}
\tag{2.3}
\end{equation*}\\
for $\mu\in \Sigma_{\theta}$. Here, $\nabla^{\perp}={}^{t}(\partial_2,-\partial_1)$. We denote by $B_{i}=-\Delta$ the Laplace operators associated with the boundary conditions in (2.2) and (2.3) for $i=1,2$, respectively. We begin with the problem (2.2).

\vspace{15pt}

\begin{prop}
Let $p\in (1,\infty)$ and $\theta\in (\pi/2,\pi)$. There exists a constant $C$ such that for $g\in L^{p}(D)$ and $\mu\in \Sigma_{\theta}$ there exists a unique solution $v\in W^{2,p}(D)$ of (2.2) satisfying 

\begin{align*}
|\mu| ||v||_{L^{p}(D)}+|\mu|^{\frac{1}{2}}||\nabla v||_{L^{p}(D)}+||\nabla^{2}v||_{L^{p}(D)}\leq C||g||_{L^{p}(D)}.   \tag{2.4}
\end{align*}\\
The operator $B_1$ is invertible on $L^{p}(D)$. 
\end{prop}

\vspace{5pt}

\begin{proof}
The assertion is proved in \cite[Theorem 3.10]{Miyakawa80}, \cite[p.404, l.8]{MiyakawaYamada92}, \cite[Theorem 1.2]{AKST} for $|\mu|\geq \delta$ and arbitrary $\delta>0$. We shall prove (2.4) for $|\mu|\leq \delta$.

We first prove the a priori estimate

\begin{align*}
||v||_{W^{1,p}(D)}\leq C||g||_{L^{p}(D)}   \tag{2.5}
\end{align*}\\
for solutions of (2.2) for $\mu=0$ by a contradiction argument. Suppose on the contrary that (2.5) were false. Then there exists a sequence of functions $\{v_m\}$ satisfying (2.2) for $\mu=0$ and $g_m$ such that 

\begin{align*}
||v_m||_{W^{1,p}(D)}=1,\quad ||g_m||_{L^{p}(D)}\to 0\quad m\to\infty.
\end{align*}\\
Since the estimate (2.4) holds for $\mu=1$, applying (2.4) for $v_m-\Delta v_m=g_m+v_m$ implies that $\{v_m\}$ is uniformly bounded in $W^{2,p}$. Thus by the Rellich-Kondrachov theorem \cite[5.7 THEOREM 1]{E}, there exists a subsequence (still denoted by $\{v_m\}$) such that $v_m$ converges to a limit $v$ in $W^{1,p}$ and the limit $v$ satisfies 

\begin{align*}
\Delta v&=0\quad \textrm{in}\ D,\\
\nabla^{\perp}\cdot v =0\quad v\cdot n&=0\quad \textrm{on}\ \partial D. 
\end{align*}\\
Since $\nabla^{\perp}\cdot v$ is harmonic and vanishes on $\partial D$, we have $\nabla^{\perp}\cdot v=0$ in $D$. Moreover, by $-\Delta v= \nabla^{\perp}(-\nabla^{\perp}\cdot v)-\nabla \D\ v$ and $v\cdot n=0$ on $\partial D$, we see that $\D\ v=0$ in $D$. Since $D$ is simply-connected, there exists a stream function $\psi$ such that $v=\nabla^{\perp}\psi$. Since $\psi$ is harmonic and constant on $\partial D$, we have $v\equiv 0$. This contradicts $||v||_{W^{1,p}}=1$. Hence (2.5) holds. By (2.5) and (2.4) for $\mu=1$, we obtain 

\begin{align*}
||v||_{W^{2,p}(D)}\leq C ||g||_{L^{p}(D)}  \tag{2.6}
\end{align*}\\
for solutions of (2.2) for $\mu=0$. 

By applying (2.6) for solutions of (2.2) for $|\mu|\leq \delta $, we obtain 

\begin{align*}
||v||_{W^{2,p}(D)}\leq C||g||_{L^{p}(D)}+C\delta ||v||_{L^{p}(D)},
\end{align*}\\
with the constant $C$ independent of $\delta$. We take a small constant $\delta>0$ such that $C\delta \leq 1/2$ and obtain $||v||_{W^{2,p}}\leq C||g||_{L^{p}}$. Thus (2.4) holds for $|\mu|\leq \delta$.
\end{proof}

\vspace{15pt}

We next estimate the resolvent of the Neumann problem (2.3). 

\vspace{15pt}

\begin{prop}
There exists a constant $C$ such that for $h\in L^{p}(D)$ and $\mu\in \Sigma_{\theta}$ there exists a unique solution $w\in W^{2,p}(D)$ of (2.3) satisfying 

\begin{align*}
|\mu| ||w||_{L^{p}(D)}+|\mu|^{\frac{1}{2}}||\nabla w||_{L^{p}(D)}+||\nabla^{2}w||_{L^{p}(D)}\leq C||h||_{L^{p}(D)}.   \tag{2.7}
\end{align*}
\end{prop}

\vspace{5pt}

\begin{proof}
Since the estimate (2.7) hold for $|\mu|\geq \delta$ and arbitrary $\delta>0$ \cite{ADN}, \cite[Theorem 3.1.2, 3.1.3]{Lunardi}, we prove it for $|\mu|\leq \delta$.

For $p=2$, integration by parts yields

\begin{align*}
|\mu| ||w||_{2}+|\mu|^{\frac{1}{2}} ||\nabla w||_{2}\leq C||h||_{2}  \tag{2.8}
\end{align*}\\
with some constant $C$, independent of $\mu$. We consider the Neumann problem  

\begin{equation*}
\begin{aligned}
-\Delta p&=f\quad \textrm{in}\  D, \\
\partial_n p&=0\quad \textrm{in}\ \partial D,
\end{aligned}
\tag{2.9}
\end{equation*}\\
for average-zero functions $f\in L^{p}$, i.e. $\int_{D}f\dd x=0$. Solutions of (2.9) uniquely exist up to an additive constant and satisfy the estimate 

\begin{align*}
||\nabla p||_{W^{1,p}(D)}\leq C||f||_{L^{p}(D)},   \tag{2.10}
\end{align*}\\
by \cite{LM}. Applying (2.10) for $-\Delta w=h-\mu w$ yields the estimate (2.7) for $p=2$.

We next consider the case $p\in (2,\infty)$. We apply the Gagliardo-Nirenberg inequality in $\mathbb{R}^{2}$ for an extension of $\varphi\in H^{1}(D)$ to $\mathbb{R}^{2}$ and observe that the inequality

\begin{align*}
||\varphi||_{L^{p}(D)}\leq C||\varphi||_{L^{2}(D)}^{\frac{2}{p}}||\varphi||_{H^{1}(D)}^{1-\frac{2}{p}}
\end{align*}\\
holds. Applying the above inequality for $\varphi=w$ and (2.8) imply

\begin{align*}
||w||_{L^{p}(D)}
\leq \frac{C'}{|\mu|}||h||_{L^{2}(D)},
\end{align*}\\
for $|\mu|\leq \delta$. Since $D$ is bounded, the right-hand side is estimated by the $L^{p}$-norm of $h$. We estimate $\nabla w$ by the same way and apply (2.10) to obtain (2.7).

For $p\in (1,2)$, a duality argument and (2.10) yield the estimate (2.7) for $|\mu|\leq \delta$. We proved (2.7) for $p\in (1.\infty)$.
\end{proof}

\vspace{15pt}
We now derive a solution formula for the problem (2.1). We use the cylindrical coordinate $x_1=r\cos\theta$, $x_2=r\sin\theta$, $x_3=z$ and decompose a vector field $f=f^{r}e_{r}(\theta)+f^{\theta}e_{\theta}(\theta)+f^{z}e_{z}$ by the basis $e_{r}(\theta)={}^{t}(\cos\theta,\sin\theta,0)$, $e_{\theta}(\theta)={}^{t}(-\sin\theta,\cos\theta,0)$, $e_{z}={}^{t}(0,0,1)$. In the sequel, we write the horizontal component by $f_{h}=f^{r}e_{r}+f^{\theta}e_{\theta}$. We define a partial Fourier transform $\hat{u}={\mathcal{F}}u$ by

\begin{align*}
({\mathcal{F}}u)(x_h,\xi)=\int_{\mathbb{R}}e^{-ix_3\xi}u(x_h,x_3)\dd x_{3},
\end{align*}\\
for functions $u(\cdot, x_3)$ in the Schwartz class $\mathit{S}(\mathbb{R}; X)$ for a Banach space $X$. See \cite[Chapter 6]{BL}.

\vspace{15pt}

\begin{prop}
For $f\in C^{\infty}_{c}(\Pi)$, solutions of (2.1) are represented by $u=u_h+u^{z}e_{z}$ and 

\begin{equation*}
\begin{aligned}
u_{\textrm{h}}&={\mathcal{F}}^{-1}(\lambda+\xi^{2}+B_1)^{-1}{\mathcal{F}}f_{h},\\
u^{z}&={\mathcal{F}}^{-1}(\lambda+\xi^{2}+B_2)^{-1}{\mathcal{F}}f^{z},
\end{aligned}
\tag{2.11}
\end{equation*}\\
where ${\mathcal{F}}^{-1}$ denotes the Fourier inverse transform. 
\end{prop}

\begin{proof}
Let $u=u_h+u^{z}e_{z}$ be a solution of (2.1). Since 

\begin{align*}
\nabla \times u=\Big(-\partial_z u^{\theta}+\frac{1}{r}\partial_{\theta}u^{z}\Big)e_{r}+\Big(\partial_z u^{r}-\partial_r u^{z} \Big)e_{\theta}
+\Big(\partial_r u^{\theta}+\frac{1}{r}u^{\theta}-\frac{1}{r}\partial_{\theta}u^{r} \Big)e_{z},
\end{align*}\\
the Neumann boundary condition in (2.1) implies 

\begin{align*}
u^{r}=0,\quad \partial_ru^{\theta}+u^{\theta}=0,\quad \partial_r u^{z}=0 \quad\textrm{on}\ \{r=1\}. 
\end{align*}\\
Thus $u_{h}$ and $u^{z}$ satisfy 

\begin{equation*}
\begin{aligned}
\lambda u_h-\Delta u_h&=f_h\quad \textrm{in}\ \Pi,\\
\nabla^{\perp}\cdot u_h=0,\ u_h\cdot n&=0\hspace{15pt} \textrm{on}\ \partial \Pi,\\
\lambda u^{z}-\Delta u^{z}&=f^{z}\hspace{12pt} \textrm{in}\ \Pi,\\
\partial_n u^{z}&=0\hspace{17pt} \textrm{on}\ \partial \Pi.
\end{aligned}
\end{equation*}\\
We consider the partial Fourier transform for $u_h$ and $u^{z}$. Since $\hat{u}_h$ and $\hat{u}^{z}$ satisfy (2.2) and (2.3) for $\mu=\lambda +\xi^{2}$, $g=\hat{f}_{h}$, $h=\hat{f}_{z}$, we see that $\hat{u}_h=(\mu+B_1)^{-1}\hat{f}_{h}$ and $\hat{u}^{z}=(\mu+B_2)^{-1}\hat{f^{z}
}$ by Propositions 2.1 and 2.2. By the Fourier inverse transform, we obtain (2.11).
\end{proof}

\vspace{15pt}

\subsection{$L^{2}$-estimates}

By using the formula (2.11), we construct unique solutions of (2.1).
\vspace{15pt}

\begin{lem}
There exist constants $\delta>0$ and $C>0$ such that for $f\in L^{p}$ and $\lambda\in  \Sigma_{\theta}$ satisfying $|\lambda|\geq \delta$, there exists a unique solution of (2.1) satisfying 

\begin{align*}
|\lambda|||u||_{L^{p}(\Pi)}+|\lambda|^{\frac{1}{2}} ||\nabla u||_{L^{p}(\Pi)}+ ||\nabla^{2} u||_{L^{p}(\Pi)}\leq C||f||_{L^{p}(\Pi)}.   \tag{2.12}
\end{align*}
\end{lem}

\vspace{15pt}
We first apply a multiplier theorem on a Hilbert space.
\vspace{15pt}

\begin{prop}
The assertion of Lemma 2.4 holds for $p=2$ and $\delta=0$.
\end{prop}

\vspace{5pt}

\begin{proof}
We prove the estimate (2.12) for solutions given by the formula (2.11) for $f\in C^{\infty}_{c}(\Pi)$. We shall show the estimate 

\begin{align*}
|\lambda|||u_h||_{L^{2}(\Pi)}+|\lambda|^{\frac{1}{2}} ||\nabla u_h||_{L^{2}(\Pi)}+ ||\nabla^{2} u_h||_{L^{2}(\Pi)}\leq C||f_h||_{L^{2}(\Pi)}   \tag{2.13}
\end{align*}\\
for $\lambda\in \Sigma_{\theta}$. We are able to estimate $u_{z}$ by using (2.7) in the same way. We set

\begin{align*}
m_{1}(\xi)=\lambda (\lambda+\xi^{2}+B_{1})^{-1}.   
\end{align*}\\
By the resolvent estimate (2.4), the operator $m_{1}(\xi)$ acts as a bounded operator on $L^{2}(D)$. We set $\lambda=re^{i\eta}$ for $\eta\in (-\theta,\theta)$. We observe that $|\lambda|/|\lambda+\xi^{2}|\leq 1$ for $|\eta|\leq \pi/2$. For $\pi/2\leq |\eta|< \theta$, it follows that 

\begin{align*}
\frac{|\lambda|}{|\lambda+\xi^{2}|}=\frac{1}{|e^{i \eta}+\frac{\xi^{2}}{r}|}
\leq \frac{1}{|\sin\theta|}, \quad \theta\in (\pi/2,\pi). 
\end{align*}\\
Thus the operators $\{m_{1}(\xi)\}\subset B(L^{2}(D))$ are uniformly bounded for $\xi\in \mathbb{R}\backslash \{0\}$ and there exists a constant $C$, independent of $\lambda$ such that 

\begin{align*}
\sup\Big\{||m_1(\xi)|| \ |\ \xi\in \mathbb{R}\backslash \{0\}  \Big\}\leq C,\qquad \lambda\in \Sigma_{\theta}.   \tag{2.14}
\end{align*}\\
Here, $||\cdot||$ denotes the operator norm on $L^{2}(D)$. Since the resolvent is holomorphic for $\mu\in \Sigma_{\theta}$ and $d(\mu+B_1)^{-1}/d\mu=-(\mu+B_1)^{-2}$, the estimate (2.14) implies 

\begin{align*}
\sup\Big\{||\xi m'_1(\xi)|| \ |\ \xi\in \mathbb{R}\backslash \{0\}  \Big\}\leq C'.  
\end{align*}\\
Thus Mikhlin's operator-valued multiplier theorem on a Hilbert space \cite[6.1.6 Theorem]{BL}, \cite[I, 3.20. Corollary]{DeHP} implies that the operator ${\mathcal{F}}^{-1}m_{1}(\cdot){\mathcal{F}}$ acts as a bounded operator on $L^{q}(\mathbb{R}; L^{2}(D))$ for all $q\in (1,\infty)$. Since $L^{2}(\mathbb{R}; L^{2}(D))=L^{2}(\Pi)$, this in particular implies the resolvent estimate 

\begin{align*}
|\lambda|||u_{h}||_{L^{2}(\Pi) }\leq C||f_{h}||_{L^{2}(\Pi)}.
\end{align*}\\
By a similar way, we are able to estimate the higher order terms and obtain (2.13). We proved (2.12) for $p=2$ and $f\in C^{\infty}_{c}(\Pi)$.

For general $f\in L^{2}(\Pi)$, we construct solutions by taking a sequence $f_{m}\in C^{\infty}_{c}(\Pi)$ such that $f_m\to f$ in $L^{2}(\Pi)$ and using the estimate (2.12). The uniqueness follows from integration by parts.
\end{proof}

\vspace{15pt}

\begin{rem}
By using a multiplier theorem on a UMD-space, we are able to obtain the  $L^{p}$-estimate 

\begin{align*}
|\lambda|||u||_{L^{p}(\Pi)}\leq C||f||_{L^{p}(\Pi)}   \tag{2.15}
\end{align*}\\
for solutions to (2.1) and $\lambda\in \Sigma_{\theta}$. We give a proof for (2.15) in Appendix B.
\end{rem}

\vspace{15pt}
\subsection{$L^{p}$-estimates}
We next prove (2.12) for $p\in (1,\infty)$ and large $|\lambda|\geq \delta$ by a cut-off function argument. We apply $L^{p}$-estimates for the resolvent equation in a smoothly bounded domain $G\subset \mathbb{R}^{3}$:

\begin{equation*}
\begin{aligned}
\lambda u-\Delta u&=f\quad \textrm{in}\ G, \\
\nabla\times u\times n=g,\ u\cdot n&=0\quad \textrm{on}\ \partial G.
\end{aligned}
\tag{2.16}
\end{equation*}

\begin{prop}
For $\delta >0$, there exists a constant $C>0$ such that 

\begin{align*}
|\lambda| ||u||_{L^{p}(G)}+|\lambda|^{\frac{1}{2}} ||\nabla u||_{L^{p}(G)}
+||\nabla^{2} u||_{L^{p}(G)}\leq C(||f||_{L^{p}(G)}+|\lambda|^{\frac{1}{2}}||g||_{L^{p}(G)}+||\nabla g||_{L^{p}(G)})  \tag{2.17}
\end{align*}\\
holds for solutions of (2.16) for $|\lambda|\geq \delta$, $f\in L^{p}(G)$ and $g\in W^{1,p}(G)$ satisfying $g\cdot n=0$ on $\partial G$.
\end{prop}

\vspace{5pt}

\begin{proof}
The stronger estimate 

\begin{align*}
|\lambda| ||u||_{L^{p}(G)}+|\lambda|^{\frac{1}{2}} ||\nabla u||_{L^{p}(G)}
+||\nabla^{2} u||_{L^{p}(G)}\leq C(||f||_{L^{p}(G)}+|\lambda|^{\frac{1}{2p'}}||g||_{L^{p}(\partial G)})  
\end{align*}\\
is proved in \cite[Theorem 1.2]{AKST}, where $p'$ denotes the conjugate exponent to $p$. Since the trace of $h$ is estimated by 

\begin{align*}
||g||_{L^{p}(\partial G)}\leq C||g||_{L^{p}(G)}^{\frac{1}{p'}}|| g||_{W^{1,p}(G)}^{\frac{1}{p}},
\end{align*}\\
by \cite[II.4., Theorem II.4.1]{Gal}, applying the Young's inequality implies (2.17).
\end{proof}

\vspace{15pt}

\begin{prop}
The a priori estimate (2.12) holds for solutions of (2.1) for $f\in C_{c}^{\infty}(\Pi)$. 
\end{prop}

\vspace{5pt}

\begin{proof}
Let $\{\varphi_j\}_{j=-\infty}^{\infty}\subset C_{c}^{\infty}(\mathbb{R})$ be a partition of the unity such that $0\leq \varphi_j\leq 1$, $\textrm{spt}\ \varphi_j\subset [j-1,j+1]$, $\sum_{j=-\infty}^{\infty}\varphi_j(x_3)=1$, $x_3\in \mathbb{R}$. For a solution $u$ of (2.1), we see that $u_j=u\varphi_j$ satisfies 

\begin{align*}
\lambda u_j-\Delta u_j&=f_j\quad \textrm{in}\ G_j, \\
\nabla\times u_j\times n=g_j,\ u\cdot n&=0\quad \textrm{on}\ \partial G_j,
\end{align*}\\
for $G_j=D\times (j-1,j+1)$ and 

\begin{align*}
f_{j}&=f\varphi_j-2\partial_{x_3}u\partial_{x_3}\varphi_j-u\partial_{x_3}^{2}\varphi_j,\\
g_j&=\nabla \varphi_j\times u\times n.
\end{align*}\\
We take a smoothly bounded domain $\tilde{G}_{j}$ such that $G_j\subset \tilde{G}_{j}\subset D\times [j-2,j+2]$. Since the estimate (2.17) holds in $\tilde{G_{j}}$ for $|\lambda|\geq \delta$ with $\delta>0$ by Proposition 2.7, we have

\begin{align*}
|\lambda| ||u_{j}||_{L^{p}(\tilde{G_{j}})}+|\lambda|^{\frac{1}{2}} ||\nabla u_{j}||_{L^{p}(\tilde{G_{j}})}
+||\nabla^{2} u_{j}||_{L^{p}(\tilde{G_{j}})}\leq C(||f_{j}||_{L^{p}(\tilde{G_{j}})}+|\lambda|^{\frac{1}{2}}||g_{j}||_{L^{p}(\tilde{G_{j}})}+||\nabla g_{j}||_{L^{p}(\tilde{G_{j}})}  ) 
\end{align*}\\
with some constant $C$, independent of $j$. The above estimate yields

\begin{align*}
|\lambda| ||u\varphi_{j}||_{L^{p}(\Pi)}+|\lambda|^{\frac{1}{2}} ||\nabla u\varphi_{j}||_{L^{p}(\Pi)}
+||\nabla^{2} u\varphi_{j}||_{L^{p}(\Pi)}
\leq C'(||f||_{L^{p}(G_{j})}
+||u||_{W^{1,p}(G_{j})}+|\lambda|^{\frac{1}{2}}||u||_{L^{p}(G_{j})} ). 
\end{align*}\\
By summing over $j$, we obtain 

\begin{align*}
|\lambda| ||u||_{L^{p}(\Pi)}+|\lambda|^{\frac{1}{2}} ||\nabla u||_{L^{p}(\Pi)}
+||\nabla^{2} u||_{L^{p}(\Pi)}
\leq C''\big(||f||_{L^{p}(\Pi)}
+\delta^{-\frac{1}{2}}(|\lambda| ||u||_{L^{p}(\Pi)}+|\lambda|^{\frac{1}{2}}||\nabla u||_{L^{p}(\Pi)}) \big)
\end{align*}\\
for $|\lambda|\geq \delta$ and $\delta\geq 1$. We take $\delta\geq 1$ so that $C''\delta^{-1/2}\leq 1/2$ and obtain (2.12).
\end{proof}

\vspace{5pt}

\begin{proof}[Proof of Lemma 2.4]
We apply the a priori estimate (2.12) for solutions given by the formula (2.11) for $f\in C^{\infty}_{c}(\Pi)$. For general $f\in L^{p}(\Pi)$, we construct solutions by an approximation by elements of $C^{\infty}_{c}(\Pi)$ and the estimate (2.12). The uniqueness follows from a duality argument. The proof is now complete.
\end{proof}

\vspace{15pt}

\section{Fractional powers}

\vspace{15pt}
In this section, we see that a domain of a square root of the Laplace operator $B_0=B+\lambda_0$ for $\lambda_0>0$ is continuously embedded to the Sobolev space $W^{1,p}(\Pi)$. We first recall the notion of a bounded $H^{\infty}$-calculus for a sectorial operator in an abstract Banach space. In the subsequent section, we apply an abstract theory to the operator $B_0$ and deduce the continuous embeddings.

\vspace{15pt}

\subsection{BIP and ${H}^{\infty}$}

We recall a bounded $H^{\infty}$-calculus \cite{McIntosh86}. We follow a booklet \cite{DeHP}. We say that a closed linear operator $L$ in a Banach space $X$ is \textit{sectorial} if the domain $D(L)$ and the range $R(L)$ are dense in $X$, $(-\infty,0)\subset \rho(L)$ and there exists a constant $C>0$ such that 

\begin{align*}
||(t+L)^{-1}||\leq \frac{C}{t}\qquad t>0.    \tag{3.1}    
\end{align*}\\
Here, $\rho(L)$ is the resolvent set of $L$, i.e., the set of all $\lambda\in \mathbb{C}$ such that $(\lambda-L): D(L)\subset X\longrightarrow X$ is invertible and $(\lambda-L)^{-1}$ acts as a bounded operator on $X$. If $(-\infty,0)\subset \rho(L)$, $(-t-L)^{-1}=-(t+L)^{-1}$ is a bounded operator for $t>0$ and we are able to define the condition (3.1) with the operator norm $||\cdot ||$ on $X$. Since $(t+L)^{-1}$ is differentiable and all derivatives are estimated by (3.1), by the Taylor expansion, we are able to extend $(t+L)^{-1}$ as an analytic function to a sector $\Sigma_{\theta}=\{ t \in \mathbb{C}\backslash \{0\}\ |\ |\textrm{arg}\ {t}|<\theta  \}$ for some $\theta\in (0,\pi/2)$ and $t(t+L)^{-1}$ is bounded in $\Sigma_{\theta}$, i.e., 

\begin{align*}
\sup\left\{ t||(t+L)^{-1}||\ \middle|\  t\in \Sigma_{\theta}  \right\}<\infty.  \tag{3.2}
\end{align*}\\
The infimum of $\phi=\pi-\theta$ for $\theta\in (0,\pi)$ such that (3.2) holds is called \textit{spectral angle} of $L$, denoted by $\phi_{L}$. If the operator $L$ is sectorial, we are able to define the fractional power $L^{\alpha}$ for $\alpha\in [0,1]$. By the notation above, we avoid writing $(-\tilde{L}
)^{\alpha}$.

Let $H(\Sigma_{\phi})$ denote the space of all holomorphic functions in $\Sigma_{\phi}$. For simplicity, we abbreviate the notation by omitting the symbol $\Sigma_{\phi}$. Let $H^{\infty}$ denote the space of all bounded and holomorphic functions in $\Sigma_{\phi}$. The space $H^{\infty}$ is equipped with the norm $|f|_{\infty}^{\phi}=\sup_{\lambda\in \Sigma_{\phi}} |f(\lambda)|$. We set $H_{0}=\bigcup_{\alpha<0,\beta<0}H_{\alpha,\beta}$ by $H_{\alpha,\beta}$, the space of all functions $f\in H$ such that $|f|_{\alpha,\beta}^{\infty}=\sup_{|\lambda|\leq 1}|\lambda|^{\alpha}|f(\lambda)|+\sup_{|\lambda|\geq 1}|\lambda|^{-\beta}|f(\lambda)|$ is finite. The space $H_{0}$ is smaller than $H^{\infty}$ and consists of functions vanishing at $\lambda=0$ and $|\lambda|\to\infty$. 

We define bounded linear operators $f(L)$ for holomorphic functions $f\in H_0$. Here, we take $\phi \in (\phi_L,\pi)$ so that the spectrum $\sigma (L)=\mathbb{C}\backslash \rho(L)$ is included in $\Sigma_{\phi}$. We take a counterclockwise integral path $\Gamma=(\infty,0]e^{i\psi}\cup [0,\infty)e^{-i\psi}$ for $\psi \in (\phi_{L},\phi)$ in $\Sigma_{\phi}$ and set a bounded linear operator $f(L)$ by the Dunford integral

\begin{align*}
f(L)=\frac{1}{2\pi i}\int_{\Gamma}f(\lambda)(\lambda-L)^{-1}\dd \lambda, \quad f\in H_0(\Sigma_{\phi}).  \tag{3.3}
\end{align*}\\
We say that the operator $L$ admits \textit{a bounded $H^{\infty}$-calculus} if there exists $\phi\in (\phi_L,\pi)$ and $K>0$ such that 

\begin{align*}
||f(L)||\leq K|f|_{\infty}^{\phi}\quad \textrm{for}\ f\in H_{0}(\Sigma_{\phi}). \tag{3.4}
\end{align*}\\
The infimum of such $\phi$ is called \textit{$H^{\infty}$-angle}, denoted by $\phi^{\infty}_{L}$. If the operator $L$ admits a bounded $H^{\infty}$-calculus, we are able to define a bounded linear operator $f(L)$ for $f\in H^{\infty}$ by an approximation. In particular, we are able to define pure imaginary powers $L^{is}$ since $f(\lambda)=\lambda^{is}$ is bounded and holomorphic in $\Sigma_{\pi}$. Here, $\lambda^{is}$ takes the principal branch. We say that the operator $L$ admits \textit{a bounded imaginary power} if there exists a constant $C$ such that 

\begin{align*}
||L^{is}||\leq C\quad \textrm{for}\ |s|\leq 1.   \tag{3.5}
\end{align*}\\
Since $L^{is}$ forms a group, the estimate (3.5) implies that $L^{is}$ is quasi-bounded, i.e., $||L^{is}||\leq Ce^{\theta |s|}$ for $s\in \mathbb{R}$ and some constants $\theta, C>0$. The infimum of such $\theta$ is called \textit{power angle} of $L$, denoted by $\theta_{L}$. It follows from (3.4) that $0\leq \phi_{L}\leq \theta_{L}\leq \phi^{\infty}_{L}<\pi$. 

If a sectorial operator admits a bounded imaginary power, it follows that the domain of the fractional power $D(L^{\alpha})$ agrees with the complex interpolation space $[X,D(L)]_{\alpha}$ for $\alpha\in [0,1]$. Here, $D(L^{\alpha})$ is equipped with the graph-norm. See \cite{Tanabe79}, \cite{DeHP} for fractional powers of a sectorial operator.

\vspace{15pt}

\subsection{A domain of a square root}

We now define fractional powers for the operator $B_0=B+\lambda_0$ for $\lambda_0>0$. By Lemma 2.4, the operator $B_0$ is invertible and sectorial on $L^{p}$ with spectral angle zero. The boundedness of pure imaginary powers of the operator $B_0$ is proved by R. Seeley in \cite{Seeley71}. More strongly, we have:

\vspace{15pt}

\begin{prop}
There exists $\lambda_0>0$ such that $B_0$ admits a bounded $H^{\infty}$-calculus. 
\end{prop}

\vspace{5pt}

\begin{proof}
The assertion is proved in \cite[Theorem 3.1]{GHT} for general uniformly regular domains.
\end{proof}

\vspace{15pt}
Proposition 3.1 implies that $D(B_0^{\alpha})=[L^{p}, D(B_0)]_{\alpha}$ for $\alpha \in [0,1]$. Here, we define the fractional power $B_0^{\alpha}$ as an inverse of the bounded operator $B_0^{-\alpha}$, i.e.,  

\begin{align*}
&B_0^{\alpha}u=(B_0^{-\alpha})^{-1}u,\quad u\in D(B_0^{\alpha}),\\
&D(B_0^{\alpha})=R(B_0^{-\alpha}), \quad \alpha\in (0,1],
\end{align*}\\
and $B_0^{0}=I$. The fractional power $B_0^{-\alpha}$ is defined by the Dunford integral (3.3) by taking $f(\lambda)=\lambda^{-\alpha}$ and the  counter-clockwise integral path $\Gamma$, consisting of the two half lines $\{\lambda \in \mathbb{C}\ |\ \textrm{arg}\ (\lambda-a)=\pm \psi\}$ for some $a>0$ and $\psi\in (0,\pi/2)$.  

\vspace{15pt}

We deduce a continuous embedding of the domain of the square root $B_0^{1/2}$, which is used later in Section 6.

\vspace{15pt}

\begin{lem}
The domain $D(B_0^{1/2})$ is continuously embedded to $W^{1,p}(\Pi)$.
\end{lem}

\vspace{15pt}

\begin{proof}
Since there exists a linear extension operator from the complex interpolation space $[L^{p}(\Pi), W^{2,p}(\Pi)]_{1/2}$ to the Sobolev space $W^{1,p}(\mathbb{R}^{3})$ \cite[Theorem 6]{Muramatsu74}, the interpolation space agrees with the Sobolev space $W^{1,p}(\Pi)$ (see also \cite[7.57]{Ad}). Thus we have

\begin{align*}
D(B_0^{1/2})
=[L^{p}(\Pi), D(B_0)]_{1/2}
\subset [L^{p}(\Pi), W^{2,p}(\Pi)]_{1/2}=W^{1,p}(\Pi),
\end{align*}\\
with continuous injection.
\end{proof}

\vspace{15pt}

\section{The Stokes operator}

\vspace{15pt}
We define the Stokes operator as a restriction of the Laplace operator in a solenoidal vector space. Since the Helmholtz projection operator is commutable with the Laplace operator subject to the Neumann boundary condition, a restriction of the semigroup $e^{-tB}$ forms a bounded $C_0$-analytic semigroup on $L^{p}_{\sigma}$.

\vspace{15pt}

\begin{lem}
\begin{align*}
&\mathbb{P}B=B\mathbb{P}\quad \textrm{on}\ D(B),   \tag{4.1}\\
&\mathbb{P}(\lambda+B)^{-1}=(\lambda+B)^{-1}\mathbb{P}\quad \textrm{on}\ L^{p}.   \tag{4.2}
\end{align*}
\end{lem}

\vspace{5pt}

\begin{proof}
We prove (4.1). The equality (4.2) follows from (4.1). We take $u\in D(B)$. Since the operator $\mathbb{P}$ acts as a bounded operator on $W^{2,p}$ \cite[Theorem 6]{ST98} (see Lemma 6.2 in Section 6), the function $\mathbb{P}u$ belongs to $W^{2,p}$. By taking the rotation to

\begin{align*}
u=\mathbb{P}u+\mathbb{Q}u=f+\nabla \phi,
\end{align*}\\
we see that $\nabla\times  f\times n=0$ on $\partial\Pi$, where $\mathbb{Q}=I-\mathbb{P}$. Thus, $\mathbb{P}u\in D(B)\cap L^{p}_{\sigma}$. We shall show that

\begin{align*}
\mathbb{Q} B\mathbb{P}u&=0,  \tag{4.3} \\
\mathbb{P}B\mathbb{Q}u&=0.  \tag{4.4}
\end{align*}\\
The property (4.1) follows from (4.3) and (4.4) since 

\begin{align*}
\mathbb{P}B u
&=\mathbb{P}B(\mathbb{P}u+\mathbb{Q}u)\\
&=(I-\mathbb{Q})B\mathbb{P}u+\mathbb{P}B\mathbb{Q}u\\
&=B\mathbb{P}u.
\end{align*}\\
We prove (4.3). The property (4.4) follows from a duality. We set

\begin{align*}
&\nabla \Phi=\mathbb{Q}B\mathbb{P}u,\\
&B\mathbb{P}u=-\Delta f,\quad f=\mathbb{P}u\in D(B)\cap L^{p}_{\sigma}.
\end{align*}\\
It is not difficult to see that $\nabla \Phi\equiv 0$ since $\Phi$ satisfies the Neumann problem

\begin{align*}
&\Delta \Phi=-\D\ (\Delta f)=0\quad \textrm{in}\ \Pi,\\
&\frac{\partial \Phi}{\partial n}=-\Delta f\cdot n=\D_{\partial\Pi}\ (\nabla\times f\times n)=0\quad \textrm{on}\ \partial \Pi,
\end{align*}\\
in a weak sense. Here, $\D_{\partial\Pi}$ denotes the surface divergence on $\partial\Pi$. Indeed, integration by parts yields 

\begin{align*}
\int_{\Pi}\Delta f\cdot \nabla \varphi\dd x
=\sum_{i,j=1}^{3}\int_{\partial\Pi}(\partial_{j}f^{i}-\partial_{i}f^{j} )n^{j}\partial_{i}\varphi \dd {\mathcal{H}}
=\int_{\partial\Pi}(\nabla\times f\times n)\cdot \nabla \varphi \dd {\mathcal{H}}=0,
\end{align*}\\
for $\varphi\in C^{\infty}_{c}(\overline{\Pi})$. Since $\nabla \varphi$ is orthogonal to solenoidal vector fields, it follows that

\begin{align*}
\int_{\Pi}\nabla \Phi\cdot \nabla \varphi\dd x
=-\int_{\Pi}\Delta f\cdot \nabla \varphi\dd x
=0.
\end{align*}\\
The above equality is extendable for all $\nabla \varphi\in G^{p'}(\Pi)$ since gradients of functions in $C^{\infty}_{c}(\overline{\Pi})$ are dense in $G^{p'}(\Pi)=\{\nabla \varphi\in L^{p'}(\Pi)\ |\ \varphi\in L^{p'}_{\textrm{loc}}(\Pi) \}$ \cite[Lemma 7]{ST98}, where $p'$ is the conjugate exponent to $p$. It follows that 

\begin{align*}
(\nabla \Phi,g)=(\nabla \Phi,\mathbb{P}g+\mathbb{Q}g)=(\nabla \Phi,\mathbb{Q}g)
=0\quad \textrm{for}\ g\in C^{\infty}_{c}.
\end{align*}\\
Here, $(f,g)$ denotes the integral of $f\cdot g$ in $\Pi$ for $f\in L^{p}$ and $g\in L^{p'}$. We proved $\nabla \Phi\equiv 0$. 
\end{proof}

\vspace{15pt}

We consider the Stokes operator

\begin{equation*}
\begin{aligned}
&A u= B u\quad u\in D(A),\\
&D(A)=D(B)\cap L^{p}_{\sigma},
\end{aligned}
\tag{4.5}
\end{equation*}\\
and the Neumann problem

\begin{equation*}
\begin{aligned}
\lambda u-\Delta u=f,\ \D\ u=0\quad \textrm{in}\ \Pi,\\
\nabla\times u\times n=0,\quad u\cdot n=0\quad \textrm{on}\ \partial\Pi.
\end{aligned}
\tag{4.6}
\end{equation*}\\
The problem (2.1) is equivalent to (4.6) for solenoidal vector fields $f\in L^{p}_{\sigma}$.

\vspace{15pt}

\begin{lem}

\begin{align*}
(\lambda+A)^{-1}=(\lambda+B)^{-1}\quad \textrm{on}\ L^{p}_{\sigma},   \tag{4.7}
\end{align*}\\
for $\lambda\in \rho(-B)$. In particular, the Stokes operator $-A$ generates a bounded $C_0$-analytic semigroup on $L^{p}_{\sigma}$.
\end{lem}

\vspace{5pt}

\begin{proof}
We set  $u=(\lambda+B)^{-1}f$ for $f\in L^{p}_{\sigma}$. It follows from (4.2) that 

\begin{align*}
\mathbb{P}u=\mathbb{P}(\lambda +B)^{-1}f=(\lambda +B)^{-1}f=u.
\end{align*}\\
Hence $u\in L^{p}_{\sigma}\cap D(B)=D(A)$. Since $u$ is a unique solution of (4.6), $u=(\lambda+A)^{-1}f$.
\end{proof}

\vspace{15pt}
We set the operator 

\begin{align*}
A_0=A+\lambda_0\quad \lambda_0>0,
\end{align*}\\
and define fractional powers of the operator by the same way as we did for $B_0$ in the previous section. Since the resolvent of $A_0$ agrees with that of $B_0$ on $L^{p}_{\sigma}$ by (4.7), we have
\vspace{15pt}

\begin{prop}
\begin{align*}
&A_0^{-\alpha}=B_{0}^{-\alpha}\quad \textrm{on}\ L^{p}_{\sigma},  \tag{4.8}\\
&R(A_{0}^{-\alpha})=R(B_{0}^{-\alpha})\cap L^{p}_{\sigma}.   \tag{4.9}
\end{align*}
\end{prop}

\vspace{5pt}

\begin{proof}
The property (4.8) follows from (4.7). We show (4.9). For an arbitrary $f\in R(A_{0}^{-\alpha})$, there exists $u\in L^{p}_{\sigma}$ such that $f=A_{0}^{-\alpha}u\in L^{p}_{\sigma}$. Since $f\in R(B_0^{-\alpha})$ by (4.8), we have $f\in R(B_0^{-\alpha})\cap L^{p}_{\sigma}$. Conversely, for $f\in R(B_0^{-\alpha})\cap L^{p}_{\sigma}$, there exists $u\in L^{p}$ such that $B_0^{-\alpha}u=f$. Since $f\in L^{p}_{\sigma}$, it follows from (4.2) that 

\begin{align*}
f=\mathbb{P}f=\mathbb{P}B_0^{-\alpha}u=B_0^{-\alpha}\mathbb{P}u.
\end{align*}\\
By multiplying $B_0^{\alpha}$ by $B_0^{-\alpha}u=B_0^{-\alpha}\mathbb{P}u$, we have $u=\mathbb{P}u\in L^{p}_{\sigma}$ and $f=B_0^{-\alpha}u=A_0^{-\alpha}u$ by (4.8). Hence $f\in R(A_0^{-\alpha})$. We proved (4.9). 
\end{proof}

\vspace{15pt}

We set the fractional power $A_0^{\alpha}u$ for $u\in D(A_{0}^{\alpha})=R(A_{0}^{-\alpha})$ as we did for $B_0$. Proposition 4.3 and Lemma 3.2 imply:

\vspace{15pt}

\begin{lem}
\begin{align*}
&A_0^{\alpha}=B_{0}^{\alpha}\quad \textrm{on}\ D(A_0^{\alpha}),\\
&D(A^{\alpha}_{0})=D(B^{\alpha}_{0})\cap L^{p}_{\sigma}.
\end{align*}\\
In particular, $D(A_{0}^{1/2})$ is continuously embedded to $W^{1,p}\cap L^{p}_{\sigma}$. 
\end{lem}

\vspace{15pt}
In order to construct mild solutions of (1.4), we prepare an estimate of the composition operator $A_0^{-1/2}\mathbb{P}\D$.
\vspace{15pt}

\begin{prop}
There exists a constant $C$ such that 

\begin{align*}
||A_{0}^{-1/2}\mathbb{P}\D\ F||_{p}\leq C||F||_{p}   \tag{4.10}
\end{align*}\\
for $F\in C^{\infty}_{c}(\Pi)$. The operator $A_{0}^{-1/2}\mathbb{P}\D$ is uniquely extendable to a bounded operator on $L^{p}$.
\end{prop}

\vspace{5pt}

\begin{proof}
We first observe that the operator $A_0=A_{0,p}$ defined on $L^{p}_{\sigma}$ satisfies

\begin{align*}
(A_{0,p}f,g)=(f,A_{0,p'}g),\quad f\in D(A_{0,p}),\ g\in D(A_{0,p'}).  \tag{4.11}
\end{align*}\\
For simplicity, we abbreviate the subscript $p$. By (4.11), we see that the same property holds for the resolvent of $A_0$ and we have

\begin{align*}
(A_{0}^{-1/2}f,g)=(f,A_{0}^{-1/2}g),\quad f\in L^{p}_{\sigma},\ g\in L^{p'}_{\sigma}. 
\end{align*}\\
For $\varphi\in C_{c}^{\infty}$, integration by parts yields 

\begin{align*}
(A_{0}^{-1/2}\mathbb{P}\D\ F, \varphi )
=(A_{0}^{-1/2}\mathbb{P}\D\ F, \mathbb{P}\varphi )
=(\D\ F, A_{0}^{-1/2}\mathbb{P}\varphi )  
=-(F, \nabla A_{0}^{-1/2}\mathbb{P}\varphi ).
\end{align*}\\
Since $D(A_0^{1/2})$ is continuously embedded to $W^{1,p'}$ by Lemma 4.4, there exists a constant $C$ such that 

\begin{align*}
||\nabla A_{0}^{-1/2}\mathbb{P}\varphi||_{p'}\leq C||\varphi||_{p'}.
\end{align*}\\
The estimate (4.10) follows from the duality.
\end{proof}

\vspace{15pt}

\section{Existence of mild solutions}
\vspace{10pt}

We construct solutions of an integral equation (1.4) by using analyticity of the Stokes semigroup. We first prepare linear estimates for an iterative argument.

\vspace{15pt}

\begin{lem}
Let $T_0>0$, $1<p\leq q<\infty$, $\alpha\in (0,1)$ and $|k|\leq 2$. Let $p, q$ satisfy $3(1/p-1/q)\leq 1$. There exist constants $C_1-C_4$ such that

\begin{align*}
||\partial_x^{k}e^{-tA}f||_{p}&\leq \frac{C_1}{t^{\frac{|k|}{2}}}||f||_{p},   \tag{5.1}\\
||e^{-tA}f||_{q}&\leq \frac{C_2}{t^{\frac{3}{2}(\frac{1}{p}-\frac{1}{q})}}||f||_{p},  \tag{5.2}\\
||e^{-tA}\mathbb{P}\D\ f||_{q}&\leq \frac{C_3}{t^{\frac{1}{2}+\frac{3}{2}(\frac{1}{p}-\frac{1}{q})}}||f||_{p},  \tag{5.3}\\
||(e^{-\rho A}-1)e^{-t A}f||_{p}
&\leq \frac{C_4 }{t^{\alpha}}\rho^{\alpha} ||f||_{p}, \tag{5.4}
\end{align*}\\
for $f\in L^{p}_{\sigma}$ and $0<t,\rho\leq T_0$. 
\end{lem}

\vspace{5pt}

\begin{proof}
The estimate (5.1) follows from the resolvent estimate (2.12). We prove (5.2). Applying the Gagliardo-Nirenberg inequality in $\mathbb{R}^{3}$ for an extension to $\varphi\in W^{1,q}(\Pi)$ \cite{Stein70} implies

\begin{align*}
||\varphi||_{L^{q}(\Pi)}\leq C||\varphi||_{L^{p}(\Pi)}^{1-\sigma}||\varphi||_{W^{1,p}(\Pi)}^{\sigma}  \tag{5.5}
\end{align*}\\
for $\sigma=3(1/p-1/q)\in [0,1]$. The estimate (5.2) follows from (5.1) by applying (5.5) for $\varphi=e^{-t A}f$.

We prove (5.3). Since the operator $A_0=A+\lambda_0$ is invertible and generates a bounded analytic semigroup on $L^{q}_{\sigma}$, we have 

\begin{align*}
||A_0^{\frac{1}{2}}e^{-tA_0}g||_{q}\leq \frac{C}{t^{\frac{1}{2}}}||g||_{q}\quad g\in L^{q}_{\sigma}.
\end{align*}\\
Since $e^{-tA}=e^{-tA_0}e^{-t \lambda_0 }$, we have $||A_0^{1/2}e^{-tA}g||_{q}\leq Ct^{-1/2}||g||_{q}$ for $t\leq T_0$. It follows from (4.10) and (5.2) that 

\begin{align*}
||e^{-tA}\mathbb{P}\D\ f||_{q}
&=||A_0^{\frac{1}{2}}e^{-tA}A_0^{-\frac{1}{2}}\mathbb{P}\D\ f||_{q}\\
&\leq \frac{C}{t^{\frac{1}{2}}}||e^{-\frac{t}{2}A}A_0^{-\frac{1}{2}}\mathbb{P}\D\ f||_{q}\\
&\leq \frac{C}{t^{\frac{1}{2}+\frac{3}{2}(\frac{1}{q}-\frac{1}{p})}}|| f||_{p}.
\end{align*}\\
Thus (5.3) holds.

It remains to show (5.4). It follows from (5.1) that  

\begin{align*}
||(e^{-\rho A}-1)e^{-t A}f||_{p}
&=\big\|\int_{0}^{\rho}\frac{\dd}{\dd s}e^{-sA}e^{-t A}f\dd s\big\|_{p}\\
&=\big\|\int_{0}^{\rho}Ae^{-(s+t)A}f\dd s\big\|_{p}\\
&\leq C||f||_{p}\int_{0}^{\rho}\frac{\dd s}{s+t} 
=C||f||_{p}\int_{0}^{\frac{\rho}{t}}\frac{\dd r}{r+1}.
\end{align*}\\
Since $r+1\geq r^{1-\alpha}$ for $r\geq 0$, the estimate (5.4) holds.
\end{proof}

\vspace{15pt}

By using the estimates (5.1)-(5.4), we construct mild solutions of (1.4). Let $C^{\gamma}((0,T]; L^{p})$ denote the space of all functions in $C^{\gamma}([\delta,T]; L^{p})$ for all $\delta\in (0,T)$.

\vspace{15pt}

\begin{thm}
For $u_0\in L^{p}_{\sigma}$ and $p\in [3,\infty)$, there exists $T>0$ and a unique solution $u\in C([0,T]; L^{p})$ of (1.4) satisfying 

\begin{align*}
t^{\frac{1}{2}}\nabla u\in C([0,T]; L^{p}),&    \tag{5.6}\\
u\in C^{\gamma}((0,T]; L^{p}),&   \tag{5.7}\\
\nabla u\in C^{\frac{\gamma}{2}}((0,T]; L^{p}),&\quad \gamma\in (0,1).  \tag{5.8}
\end{align*}\\
For $p=3$, $t^{3/2(1/3-1/q)}u\in C([0,T]; L^{q})$ vanishes at time zero for $q\in (3,\infty)$.
\end{thm}

\vspace{5pt}

\begin{proof}
We prove the case $p=3$. We are able to prove the case $p\in (3,\infty)$ by a similar way. We set a sequence $\{u_j\}$ by 

\begin{equation*}
\begin{aligned}
u_{j+1}&=e^{-tA}u_0-\int_{0}^{t}e^{-(t-s)A}\mathbb{P}\D\ u_{j}u_{j}\dd s,\\
u_1&=e^{-tA}u_0.
\end{aligned}
\end{equation*}\\
For $q\in (3,\infty)$ and $\gamma =3/2(1/3-1/q)$, we set 

\begin{align*}
K_{j}=\sup_{0\leq t\leq T}t^{\gamma} ||u_j||_{q}(t).
\end{align*}\\
We take $r=q/2$. Since $3(1/r-1/q)=3/q< 1$, it follows from (5.1) and (5.3) that 

\begin{align*}
||u_{j+1}||_{q}
&\leq ||e^{-tA}u_0||_{q}+\int_{0}^{t}\frac{C}{(t-s)^{\frac{1}{2}+\frac{3}{2}(\frac{1}{r}-\frac{1}{q} ) }}||u_j||_{2r}\dd s\\
&\leq ||e^{-tA}u_0||_{q}+\frac{C'}{t^{\gamma}}K_j^{2}.
\end{align*}\\
We have $K_{j+1}\leq K_1+C_0K_{j}^{2}$. Since $e^{-t A}$ is strongly continuous on $L^{3}_{\sigma}$, we see that $K_1\to 0$ as $T\to0$. We take small $T>0$ so that $K_1\leq (4C_0)^{-1}$ and obtain $K_{j}\leq 2K_1$ for all $j\geq 1$. 

By estimating $u_{j+1}-u_j$ by a similar way, we are able to show that 

\begin{align*}
\lim_{j\to \infty}\sup_{0\leq t\leq T}t^{\gamma} ||u_{j+1}-u_{j}||_{q}(t)=0.
\end{align*}\\
Thus a limit $u$ satisfies the integral equation (1.4) such that $t^{\gamma}u\in C([0,T]; L^{q})$ and $t^{\gamma}u$ vanishes at time zero. 

We show continuity at time zero. We set $K=\sup_{0\leq t\leq T}t^{\gamma}||u||_{q}(t)$ and $r=q/2$. It follows from (5.3) that 

\begin{align*}
||u-e^{-tA}u_0||_{3}
&\leq \int_{0}^{t}||e^{-(t-s)A}\mathbb{P}\D\ u u||_{3}\dd s\\
&\leq \int_{0}^{t}\frac{C}{(t-s)^{\frac{1}{2}+\frac{3}{2}(\frac{1}{r}-\frac{1}{3}) }s^{2\gamma} }  \dd s K^{2}=C'K^{2}\to 0\quad \textrm{as}\ T\to0.
\end{align*}\\
Thus $u\in C([0,T]; L^{3})$. By a similar way, $t^{1/2}\nabla u\in C([0,T]; L^{3})$ follows. We obtain the H\"older continuity (5.7) and (5.8) by estimating $u(t)-u(\tau)$ by using (5.4). The proof is complete.
\end{proof}

\vspace{10pt}

\section{Higher regularity}
\vspace{10pt}

We prove Theorem 1.1. It remains to show that mild solutions constructed in Theorem 5.2 are smooth in $\overline{\Pi}\times (0,T]$. We use the fractional power of $A_0=A+\lambda_0$, defined in Section 4. By multiplying $e^{-\lambda_0 t}$ by the mild solution $u$, we see that $v=e^{-tA_0}u$ satisfies 

\begin{equation*}
\begin{aligned}
&v=e^{-tA_0}v_0-\int_{0}^{t}e^{-(t-s)A_0}F(s)\dd s,\\   
&F=e^{\lambda_0 t}\mathbb{P}v\cdot \nabla v.
\end{aligned}
\tag{6.1}
\end{equation*}

\vspace{15pt}
Our goal is to prove:
\vspace{15pt}

\begin{thm}
All derivatives of $v$ belong to $C^{\mu}((0,T]; L^{p})$ for $\mu\in (0,1/2)$ and $p\in (3,\infty)$. In particular, $v$ is smooth in $\overline{\Pi}\times (0,T]$.
\end{thm}

\vspace{5pt}
\begin{proof}[Proof of Theorem 1.1]
The assertion follows from Theorem 6.1
\end{proof}

\vspace{15pt}
We prove Theorem 6.1 by using regularity results for linear operators.

\vspace{15pt}

\begin{lem}
\noindent
(i) Let $f\in L^{1}(0,T; L^{p})\cap C^{\mu}((0,T]; L^{p})$ for $p\in (1,\infty)$ and $\mu\in (0,1)$. Then, 

\begin{align*}
w=\int_{0}^{t}e^{-(t-s)A_0}f(s)\dd s
\end{align*}\\
belongs to $ C^{1+\mu}((0,T]; L^{p})\cap C^{\mu}((0,T]; D(A_0))$. \\
\noindent
(ii) For an integer $m\geq 0$, there exists a constant $C$ such that 

\begin{align*}
||\mathbb{P}f||_{W^{m,p}}+||\mathbb{Q}f||_{W^{m,p}}\leq C||f||_{W^{m,p}},\quad  f\in W^{m,p}.   \tag{6.2}
\end{align*}\\
\noindent 
(iii) There exists a constant $C$ such that 

\begin{align*}
||w||_{W^{m+2,p}}\leq C||A_0 w||_{W^{m,p}}   \tag{6.3}
\end{align*}\\
for $w\in D(A_0)$ such that $A_0u\in W^{m,p}$.
\end{lem}

\vspace{5pt}
\begin{proof}
The assertion (i) holds for generators of an analytic semigroup in an abstract Banach space. See \cite[4.3.1 Theorem 4.3.4]{Lunardi}. The regularity estimate (ii) is proved in \cite[Theorem 6]{ST98}. The assertion (iii) follows from Lemmas 2.4 and A.1 in Appendix A.
\end{proof}

\vspace{15pt}
We successively apply the regularity results (i)-(iii) and prove that all derivatives of $v$ belong to the H\"older space $C^{\mu}((0,T]; L^{p})$ for $p\in (3,\infty)$. Since solutions of (6.1) for $v_0\in L^{3}$ belong to $L^{p}$ for positive time, by taking some $t_0>0$ as an initial time, we may assume that initial data is in $L^{p}$. Since $F\in C^{\mu}((0,T]; L^{p})$ for $\mu\in (0,1/2)$ by (5.8), applying Lemma 6.2 (i) implies 

\begin{align*}
v\in C^{1+\mu}((0,T]; L^{p})\cap C^{\mu}((0,T]; W^{2,p} ). \tag{6.4}
\end{align*}\\
Hence $F\in C^{\mu}((0,T]; W^{1,p})$ by (6.2). We show that third derivatives of $v$ belong to $C^{\mu}((0,T]; L^{p})$ by applying Lemma 6.2 (i) and (iii). 

\vspace{15pt}

\begin{prop}
\begin{align*}
&\partial_t v\in C^{\mu}((0,T]; D(A_0^{1/2})),        \tag{6.5}   \\ 
&v\in C^{1+\mu}((0,T]; W^{1,p})\cap C^{\mu}((0,T]; W^{3,p}).    \tag{6.6}
\end{align*}
\end{prop}

\vspace{5pt}

\begin{proof}
The property (6.6) follows from (6.5) by the continuous injection from $D(A^{1/2}_0)$ to $W^{1,p}$ and the elliptic regularity estimate (6.3) for $A_0v=\partial_t v+F$. We prove (6.5). We differentiate $v$ by the fractional power $A_0^{1/2}$ and apply Lemma 6.2 (i). We use the integral representation for $t\geq \delta >0$ of the form

\begin{align*}
v(t)=e^{-(t-\delta)A_0}v(\delta)-\int_{\delta}^{t}e^{-(t-s)A_0}F(s)\dd s.  
\end{align*}\\
The first term is smooth for $t> \delta$. We multiply $A_0^{1/2}$ by the second term and observe that 

\begin{align*}
A^{1/2}_{0}\int_{\delta}^{t}e^{-(t-s)A_0}F(s)\dd s
&=\int_{\delta}^{t}A_0e^{-(t-s)A_0}A^{-1/2}_{0}F(s)\dd s\\
&=\int_{\delta}^{t}\frac{\dd}{\dd s}(e^{-(t-s)A_0}A^{-1/2}_{0}F(s))\dd s
-\int_{\delta}^{t}e^{-(t-s)A_0}A^{-1/2}_{0}F'(s)\dd s\\
&=A^{-1/2}_{0}F(t)-e^{-(t-\delta)A_{0}}A^{-1/2}_{0}F(\delta) 
-\int_{\delta}^{t}e^{-(t-s)A_0}A^{-1/2}_{0}F'(s)\dd s.
\end{align*}\\
The first two terms belong to $C^{1+\mu}((0,T]; L^{p})$ since $A^{-1/2}_{0}F'\in C^{\mu}((0,T]; L^{p} )$ by $\partial_t v v\in C^{\mu}((0,T]; L^{p} )$ and Proposition 4.5. The last term belongs to $C^{\mu}((\delta,T]; L^{p} )$ by applying Lemma 6.2 (i). We proved (6.5).
\end{proof}

\vspace{15pt}
The regularity property for third derivatives (6.6) implies that $F\in C^{1+\mu}((0,T]; L^{p})\cap C^{\mu}((0,T]; W^{2,p})$. We apply Lemma 6.2 (i) for $v^{(1)}=\partial_t v$ and obtain regularity of fourth order derivatives. 

\vspace{15pt}


\begin{prop}
\begin{align*}
\partial_t^{s}v\in C^{\mu}((0,T]; W^{4-2s,p})\quad s=0,1,2.   \tag{6.7}
\end{align*}
\end{prop}

\vspace{5pt}

\begin{proof}
By differentiating $\partial_t v+A_0v=F$ by time and integrating for $t\geq \delta >0$, we see that

\begin{align*}
v^{(1)}(t)=e^{-(t-\delta)A_0}v^{(1)}(\delta)-\int_{\delta}^{t}e^{-(t-s)A_0}F^{(1)}(s)\dd s.
\end{align*}\\
Since the right-hand side belongs to $C^{1+\mu}((0,T]; L^{p})\cap C^{\mu}((0,T]; W^{2,p})$ by applying Lemma 6.2 (i), the desired property (6.7) holds for $s=1,2$. The case $s=0$ follows by applying the higher regularity estimate (6.3) for $A_0v=-\partial_t v+F$.
\end{proof}

\vspace{15pt}

\begin{proof}[Proof of Theorem 6.1]
We prove

\begin{align*}
\partial_{t}^{s}v=v^{(s)}\in C^{\mu}((0,T]; W^{2k-2s,p})\quad s=0, 1,\cdots,k,   \tag{6.8}
\end{align*}\\
for any integers $k$. We prove by induction. For $k=2$, the assertion holds by Proposition 6.4. Suppose that (6.8) holds for some $k$. It suffices to show that 

\begin{align*}
v^{(s)}\in C^{\mu}((0,T]; W^{2k+2-2s,p})\quad s=0, 1,\cdots,k+1.   \tag{6.9}
\end{align*}\\
We prove (6.9) for all $s$ by induction. Since $v^{(k-1)}\in C^{1+\mu}((0,T]; L^{p})\cap C^{\mu}((0,T]; W^{2,p})$ by (6.8), by applying the same argument as in the proof of Propositions 6.3 and 6.4, we obtain (6.9) for $s=k+1$.

We suppose that (6.9) holds for $l+1\leq s\leq k+1$. Our goal is to prove (6.8) for $s=l$. By the assumption of our induction, we have

\begin{align*}
\partial_t v^{(l)}\in C^{\mu}((0,T]; W^{2k-2l,p}).  \tag{6.10}
\end{align*}\\
It follows from (6.8) and (6.2) that

\begin{align*}
&v^{(l)}\in C^{\mu}((0,T]; W^{2k-2l,p}),\\
&F^{(l)}\in C^{\mu}((0,T]; W^{2k-2l-1,p} ).
\end{align*}\\
Applying (6.3) for $A_0v^{(l)}=-\partial_t v^{(l)}+F^{(l)}$ implies that 

\begin{align*}
v^{(l)}\in C^{\mu}((0,T]; W^{2k-2l+1,p}).  
\end{align*}\\
Since $F^{(l)}\in C^{\mu}((0,T]; W^{2k-2l,p} )$ by (6.2), we obtain $v^{(l)}\in C^{\mu}((0,T]; W^{2k-2l+2,p})$ by (6.3). Thus (6.9) holds for $s=l$. We proved (6.9) for all $s$. The proof is now complete.
\end{proof}

\vspace{15pt}

\begin{rems}
\noindent
(i) (Regularity up to time zero)
In Theorem 1.1, if in addition that $u_0\in D(A_{q}^{m})$ for all $m\geq 0$, all derivatives of $u$ belong to $C^{\mu}([0,T]; L^{q})$ for $\mu\in (0,1/2)$ and $q\in (3,\infty)$.

\noindent
(ii) (Dirichlet boundary condition)
Theorem 1.1 holds also for the Dirichlet boundary condition. It is proved by Farwig-Ri \cite{FR07c} that the Stokes operator subject to the Dirichlet boundary condition generates a bounded $C_0$-analytic semigroup on $L^{p}_{\sigma}$ for $p\in (1,\infty)$. Moreover, the operator admits a bounded $H^{\infty}$-calculus. See also \cite{FR07a}, \cite{FR07b}, \cite{FR08}. The higher regularity estimate (6.3) is also known for the Stokes operator subject to the Dirichlet boundary condition \cite{Catta}, \cite[IV]{Gal}. See also \cite[III.1.5.1 Theorem]{Sohr}.
\end{rems}

\vspace{15pt}

\appendix
\section{Higher regularity estimates for the Laplace operator}

\vspace{15pt}

In Appendix A, we prove higher regularity estimates for the Neumann problem

\begin{equation*}
\begin{aligned}
-\Delta u&=f\quad \textrm{in}\ \Pi,\\
\nabla\times u\times n=g,\quad u\cdot n&=0\quad \textrm{on}\ \partial\Pi.
\end{aligned}
\tag{A.1}
\end{equation*}

\vspace{15pt}
Let $W^{1,p}_{\textrm{tan}}(\Pi)$ denote the space of all functions $g\in  W^{1,p}(\Pi)$ such that tangential components of $g$ vanish on the boundary $\partial\Pi$. 
\vspace{15pt}

\begin{lem}
Let $\Pi$ be the infinite cylinder. Let $u\in W^{2,p}(\Pi)$ be a solution of (A.1) for $f\in L^{p}(\Pi)$ and $g\in W^{1,p}_{\textrm{tan}}(\Pi)$ for $p\in (1,\infty)$. Assume that $f\in W^{m,p}(\Pi)$ and $g\in W^{m+1,p}(\Pi)$ for $m\geq 1$. Then, $u$ belongs to $W^{m+2,p}(\Pi)$ and the estimate 

\begin{align*}
||u||_{W^{m+2,p}(\Pi)}\leq C(||f||_{W^{m,p}(\Pi)}+||g||_{W^{m+1,p}(\Pi)}+||u||_{W^{1,p}(\Pi)} )   \tag{A.2}
\end{align*}\\
holds. 
\end{lem}

\vspace{15pt}
We prove Lemma A.1 by a reduction to bounded domains.
\vspace{15pt}

\begin{prop}
For smoothly bounded domains, the assertion of Lemma A.1 holds.
\end{prop}

\vspace{5pt}
\begin{proof}
For bounded domains, the estimate (A.2) is proved in \cite[Theorem 1.2]{AKST} by a reduction to a bent half space. 
\end{proof}

\vspace{5pt}

\begin{proof}[Proof of Lemma A.1]
We prove by a cut-off function argument as we did in the proof of Proposition 2.8. Let $\{\varphi_{j}\}_{j=-\infty}^{\infty}\subset C^{\infty}_{c}(\mathbb{R})$ be a partition of the unity such that $0\leq \varphi_j\leq 1$, $\textrm{spt}\ \varphi_j\subset [j-1,j+1]$, $\sum_{j=-\infty}^{\infty}\varphi_j(x_3)=1$, $x_3\in \mathbb{R}$. We set $u_j=u\varphi_j$ and $G_j=D\times (j-1,j+1)$. Observe that $u_j$ satisfies 

\begin{align*}
-\Delta u_j&=f_j\quad \textrm{in}\ G_j,\\
\nabla\times u_j\times n=g_j,\ u_j\cdot n&=0\quad \textrm{on}\ \partial G_j.
\end{align*}\\
for $f_j=f\varphi_j-2\partial_{x_3}u\partial_{x_3}\varphi_j-u\partial_{x_3}^{2}\varphi_{j}$ and $g_j=g\varphi_j+\nabla \varphi_j\times u\times n$. We take a smoothly bounded domain $\tilde{G}_j$ such that $G_j\subset \tilde{G}_j\subset \Pi$ and apply Proposition A.2 to estimate 

\begin{align*}
||u_j||_{W^{m+2,p}(\tilde{G}_j)}\leq C(||f_j||_{W^{m,p}(\tilde{G}_j)}+||g_j||_{W^{m+1,p}(\tilde{G}_j)}+||u_j||_{W^{1,p}(\tilde{G}_j)} ).   
\end{align*}\\
This implies

\begin{align*}
||\partial_{x}^{k}u\varphi_j||_{L^{p}(G_j)}\leq C(||f||_{W^{m,p}(G_j)}+||g||_{W^{m+1,p}(G_j)}+||u||_{W^{m+1,p}(G_j)} ).
\end{align*}\\
for $|k|=m+2$. By summing over $j$, we have

\begin{align*}
||u||_{W^{m+2,p}(\Pi)}\leq C(||f||_{W^{m,p}(\Pi)}+||g||_{W^{m+1,p}(\Pi)}+||u||_{W^{m+1,p}(\Pi)} ).
\end{align*}\\
We obtain the desired estimate by induction for $m\geq 0$. 
\end{proof}

\if{

\vspace{15pt}

\begin{prop}
When $\Pi=\mathbb{R}^{3}_{+}$, the estimate 

\begin{align*}
||\nabla^{2}u||_{L^{p}(\mathbb{R}^{3}_{+})}
\leq C(||f||_{L^{p}(\mathbb{R}^{3}_{+})}+||g||_{W^{1,p}(\mathbb{R}^{3}_{+})}  )  \tag{A.3}
\end{align*}\\
holds for solutions of (A.1). Moreover, the estimate (A.1) holds for $m\geq 1$.
\end{prop}

\vspace{5pt}

\begin{proof}
For a half space, the Neumann problem (A.1) is equivalent to the Neumann and Dirichlet problems for horizontal and vertical components, 

\begin{align*}
-\Delta u^{i}=f^{i},\quad -\Delta u^{3}&=f^{3}\quad \textrm{in}\ \mathbb{R}^{3}_{+}, \\
-\partial_3 u^{i}=g^{i},\quad u_3&=0\quad \textrm{on}\ \partial\mathbb{R}^{3}_{+},\quad i=1,2.
\end{align*}\\
Let $e^{sA}$ denote the Poisson semigroup and $A=(-\Delta_h)^{1/2}$. We denote the horizontal gradient by $\nabla_h=(\partial_1,\partial_2)$. We set a solution of the Neumann problem of the Laplace equation by

\begin{align*}
u^{i}_{1}=\int_{x_3}^{\infty}e^{sA}g^{i}\dd s.
\end{align*}\\
Since $u^{i}_{2}=u^{i}-u^{i}_{1}$ and $u^{3}$ satisfy the Poisson equations with the homogeneous boundary condition, by a reflection to $\mathbb{R}^{3}$, we obtain 

\begin{align*}
||\nabla^{2}u^{i}_{2}||_{L^{p}(\mathbb{R}^{3}_{+})}
+||\nabla^{2}u_{3}||_{L^{p}(\mathbb{R}^{3}_{+})}\leq C||f||_{L^{p}(\mathbb{R}^{3}_{+})}.
\end{align*}\\
We estimate $u^{i}_{1}$ by using the estimate of the Dirichlet problem

\begin{align*}
||\nabla e^{sA}g^{i}||_{L^{p}(\mathbb{R}^{3}_{+})}\leq C|g|_{W^{1-1/p,p}(\partial\mathbb{R}^{3}_{+})}.
\end{align*}\\
See \cite[II.11]{Gal}. Here, $|\cdot |_{W^{1-1/p,p}}$ denotes the semi-norm,

\begin{align*}
|g|_{W^{1-1/p,p}(\partial\mathbb{R}^{3}_{+})}=\int_{\partial\mathbb{R}^{3}_{+}}\int_{\partial\mathbb{R}^{3}_{+}}\frac{|g(x)-g(y)| }{|x-y|^{p-1}}\dd x\dd y.
\end{align*}\\
Since $\partial_3 u^{i}_{1}=-e^{x_3A}g^{i}$, applying the above estimate yields 

\begin{align*}
||\nabla \partial_3 u^{i}||_{L^{p}(\mathbb{R}^{3}_{+})}\leq C|g^{i}|_{W^{1-1/p,p}(\partial\mathbb{R}^{3}_{+})}.
\end{align*}\\
We estimate horizontal derivatives by using a fractional power estimate 

\begin{align*}
||\nabla_h \varphi||_{L^{p}(\mathbb{R}^{2})}\leq C||A \varphi||_{L^{p}(\mathbb{R}^{2})},\quad \varphi\in W^{1,p}(\mathbb{R}^{2}).
\end{align*}\\
By $Au^{i}_1=\partial_3 u^{i}_1$ and taking $\varphi=\nabla_h u^{i}_1$, it follows that 

\begin{align*}
||\nabla_{h}^{2} u^{i}_{1}||_{L^{p}(\mathbb{R}^{3}_{+})}
\leq C||A\nabla_h u^{i}_1||_{L^{p}(\mathbb{R}^{3}_{+})}
=C||\nabla_h\partial_3u^{i}_1||_{L^{p}(\mathbb{R}^{3}_{+})}
\leq C'|g^{i}|_{W^{1-1/p,p}(\partial\mathbb{R}^{3}_{+})}.
\end{align*}\\
Since a trace operator acts as a bounded operator from $W^{1,p}(\mathbb{R}^{3}_{+})$ to $W^{1-1/p,p}(\partial\mathbb{R}^{3}_{+})$ \cite[Theorem II, 10.2]{Gal}, we obtain (A.3). The higher regularity estimate (A.1) for $m\geq 1$ follows by estimating horizontal derivatives by using difference quotients.
\end{proof}

\vspace{15pt}
We prove Lemma A.2 for a bent half space $\Pi=\{(x_h,x_3)\in \mathbb{R}^{3}\ |\ x_3>\omega(x_h) \}$. Let $BC^{m+2}(\mathbb{R}^{2})$ denote the space of all bounded and continuous functions up to $(m+2)$-th orders in $\mathbb{R}^{2}$ for $m\geq 0$. Let $BC^{m+2,1}(\mathbb{R}^{2})$ denote the space of all functions $\omega\in BC^{m+2}(\mathbb{R}^{2})$ such that $(m+3)$-th derivatives of $\omega$ are essentially bounded.

\vspace{15pt}

\begin{prop}
Let $\Pi$ be a bent half space for $\omega\in BC^{m+2,1}(\mathbb{R}^{2})$ and $m\geq 0$. There exists a constant $\delta>0$ such that for $f\in W^{m,p}(\Pi)$ and $g\in W^{m+1,p}(\Pi)$, solutions of (A.1) belong to $W^{m+2,p}(\Pi)$ and the estimate (A.1) holds with some constant $C$, provided that $||\nabla_h \omega||_{\infty}\leq \delta$.
\end{prop}

\vspace{5pt}

\begin{proof}
We reduce the problem to a half space by changing the variable $x\in \Pi$ to

\begin{align*}
y_h&=x_h,\\
y_3&=x_3-\omega(x_h).
\end{align*}\\
We set $\tilde{u}(y)=u(x)$ and $\tilde{f}$, $\tilde{g}$ by the same way. Observe that $\tilde{u}$ satisfies

\begin{align*}
-\Delta \tilde{u}=\tilde{f}-\Delta_h\omega \partial_3 \tilde{u}+(-2\nabla_h\omega\cdot \nabla_h\partial_3\tilde{u}+|\nabla_h\omega|^{2}\partial_3^{2}\tilde{u})\quad &\textrm{in}\ \mathbb{R}^{3}_{+}.\\
\nabla\times \tilde{u}\times n=\tilde{g}+\tilde{\nabla}_{h}\omega\times \partial_3\tilde{u}\times n,\quad \tilde{u}\cdot n=0\quad &\textrm{on}\  \partial\mathbb{R}^{3}_{+}.
\end{align*}
Here, $n=n(y_h)$ is the unit outward normal vector field on $\partial \Pi$ and $\tilde{\nabla}_h=(\partial_1,\partial_2,0)$. We set $V=(v_h,v_3)$ by 

\begin{align*}
v_h&=\tilde{u}_{h}, \\
v_{3}&=\tilde{u}_{3}-\nabla_h\omega\cdot \tilde{u}_h,
\end{align*}\\
so that $V$ satisfies the Neumann problem in a half space

\begin{align*}
-\Delta V&=F\quad \textrm{in}\ \mathbb{R}^{3}_{+}, \\
\nabla\times V\times (-e_3)=G,\quad V\cdot (-e_3)&=0\quad \textrm{on}\ \partial\mathbb{R}^{3}_{+},
\end{align*}\\
for $F=(F_h,F_{3})$ and $G$ given by 

\begin{align*}
F_{h}&=\tilde{f}_{h}-\Delta_h\omega \partial_3 v_h-2\nabla_h\omega\cdot \nabla_h\partial_3v_h+|\nabla_h\omega|^{2}\partial_3^{2}v_h,\\
F_{3}&=\tilde{f}_{3}-\Delta_h \omega\partial_3 (v^{3}+\nabla_h\omega\cdot v_h)-2\nabla_h\omega\cdot \nabla_h \partial_3 (v_3+\nabla_h\omega\cdot v_h)
+|\nabla_h\omega|^{2}\partial_3^{2} (v^{3}+\nabla_h\omega\cdot v_h)    \\
&+\nabla_h(\Delta_h\omega)\cdot \nabla_h v_h+2\nabla^{2}_{h}\omega\nabla v_h+\nabla_h \omega \cdot \Delta v_h,\\
G&=\tilde{g}-\tilde{\nabla}_{h}\omega\times \partial_3 v_h\times e_3
-(\nabla \times (e_3\nabla_h\omega\cdot v_h))\times (\tilde{\nabla}_{h}\omega-e_3)
\end{align*}\\
Observe that $\nabla_h\omega$ is in front of second derivatives of $V$ in $F$ and $\nabla G$. Since $\omega\in BC^{2,1}(\mathbb{R}^{2})$, the function $\omega$ is bounded up to third derivatives. Hence there exists a constant $C$ such that 

\begin{align*}
||F||_{L^{p}(\mathbb{R}^{3}_{+})}+||G||_{W^{1,p}(\mathbb{R}^{3}_{+})}\leq C(||\tilde{f}||_{L^{p}(\mathbb{R}^{3}_{+})}+||\tilde{g}||_{W^{1,p}(\mathbb{R}^{3}_{+})}+||V||_{W^{1,p}(\mathbb{R}^{3}_{+})}+||\nabla_h\omega ||_{L^{\infty}(\mathbb{R}^{2})}||\nabla^{2}V||_{L^{p}(\mathbb{R}^{3}_{+})}  ).
\end{align*}\\
We take a constant $\delta>0$ such that $||\nabla_h\omega||_{\infty}\leq \delta$ and apply Proposition A.3 to estimate

\begin{align*}
||\nabla^{2}V||_{L^{p}(\mathbb{R}^{3}_+)}\leq C(||\tilde{f}||_{L^{p}(\mathbb{R}^{3}_{+})}+||\tilde{g}||_{W^{1,p}(\mathbb{R}^{3}_{+})}+||V||_{W^{1,p}(\mathbb{R}^{3}_{+})}+\delta||\nabla^{2}V||_{L^{p}(\mathbb{R}^{3}_{+})}  ),
\end{align*}\\
with some constant $C$, independent of $\delta$. The last term is absorbed into the left hand-side for small $\delta>0$ such that $C\delta\leq 1/2$. By changing variables, we are able to estimate the second derivatives of $u$ in $\Pi$ and obtain (A.4) for $m=0$.
\end{proof}

\vspace{5pt}

\begin{proof}[Proof of Lemma A.2]
We consider a localization. Since the boundary $\partial \Pi$ is smooth, for $x_0\in \partial\Pi$ there exist constants $\alpha,\beta>0$ and a function $x_3=\omega(x_h)$ such that 

\begin{align*}
\Pi\cap U(x_0)&=\{ (x_h,x_3)\in \mathbb{R}^{3}\ |\ \omega(x_h)<x_3<\omega(x_h)+\beta,\ |x_h-x_{0,h}|<\alpha  \},\\
\textrm{for}\ U(x_0)&=\{(x_h,x_3)\in \mathbb{R}^{3}\ |\ \omega(x_h)-\beta<x_3<\omega(x_h)+\beta,\ |x_h-x_{0,h}|<\alpha \}.
\end{align*}\\
Since $\Pi$ is bounded, we are able to take such $\alpha,\beta$, independently of $x_0$. We show that there exists a constant $r_0$ such that 

\begin{align*}
||u||_{W^{2,p}(B_{x_0}(r_0)\cap \Pi) }\leq C(||f||_{L^{p}(\Pi)}+||g||_{W^{1,p}(\Pi)}  +||u||_{W^{1,p}(\Pi)}),\quad x_0\in \partial\Pi. \tag{A.4}
\end{align*}\\
The estimate (A.1) for $m=0$ follows from a covering argument.

We may assume that $x_0=0$ and $\nabla_{h}\omega(0)=0$ by translation and rotation. We extend $x_3=\omega(x_h)$ for $x_h\in \mathbb{R}^{2}$ by the zero extension. We mollify it and denote by $\tilde{\omega}$. We take a cut-off function $\theta$ such that $\theta\equiv 1$ in $B_0(r_0)$ and $\theta\equiv 0$ in $\overline{B_{0}(2r_0)}^{c}$. We take a constant $r_0>0$ so that $B_0(2r_0)\subset \Pi\cap U(0)$. We set $\tilde{u}=u\theta$ and $\tilde{\Pi}=\{x_3>\tilde{\omega}(x_h) \}$. Since $\tilde{u}$ vanishes outside of $B_0(2r_0)$, it satisfies 

\begin{align*}
-\Delta \tilde{u}&=\tilde{f}\quad \textrm{in}\ \tilde{\Pi}, \\
\nabla\times \tilde{u}\times n=\tilde{g},\quad \tilde{u}\cdot n&=0\quad \textrm{on}\ \partial\tilde{\Pi},
\end{align*}\\
for $\tilde{f}=f\theta-2\nabla u\cdot \nabla \theta-u\Delta \theta$ and $\tilde{g}=g\theta+u\times \nabla \theta\times n$. Since $\nabla_{h}\omega(0)=0$, the norm $||\nabla_h\omega||_{\infty}$ is smaller than the constant $\delta$ in Proposition A.4 by letting $r_0>0$ sufficiently small. Hence we obtain (A.4). The proof is complete.
\end{proof}

}\fi

\vspace{15pt}

\section{$L^{p}$-resolvent estimates near $\lambda=0$}

\vspace{15pt}
In Appendix B, we prove the resolvent estimate (2.15). We apply a multiplier theorem on a UMD-space due to L.Weis \cite{Weis01}.

\vspace{15pt}
\begin{lem}
Let $p\in (1,\infty)$ and $\theta\in (\pi/2,\pi)$. There exists a constant $C$ such that 

\begin{align*}
|\lambda|||u||_{L^{p}(\Pi)}\leq C ||f||_{L^{p}(\Pi)}   \tag{B.1}
\end{align*}\\
holds for solutions of (2.1) for $f\in L^{p}$ and $\lambda\in \Sigma_{\theta}$. 
\end{lem}

\vspace{15pt}

We prove Lemma B.1 by using the solution formula (2.11). We show that resolvent of the Laplace operators $B_{i}$ ($i=1,2$) are $R$-bounded. We recall the notion of $R$-bounded. See \cite{DeHP}. Let $X$ and $Y$ be Banach spaces. Let $B(X,Y)$ denote the space of bounded linear operators from $X$ to $Y$. We say that a family of bounded linear operators $\tau\subset B(X,Y)$ is \textit{$R$-bounded} if there exists a constant $C$ such that for all $T_1,\cdots, T_N\in\tau$, $x_1,\cdots,x_N\in X$ and $N\geq 1$,

\begin{align*}
\int_{0}^{1}\Big\|\sum_{j=1}^{N}r_{j}(t)T_{j}x_j\Big\|_{Y}\dd t 
\leq C \int_{0}^{1}\Big\|\sum_{j=1}^{N}r_{j}(t)x_j\Big\|_{X}\dd t, 
\end{align*}\\
holds, where $\{r_j\}$ is a sequence of independent symmetric $\{-1,1\}$-valued random variables on $[0,1]$, e.g., the Rademacher functions $r_j(t)=\textrm{sign}(\sin(2^{j}\pi t))$. The smallest constant $C$ such that the above inequality holds is denoted by $R(\tau)$. For two families of $R$-bounded operators $\tau, \kappa\subset B(X,Y)$, the sum and product $\tau+\kappa=\{T+K\ |\ T\in \tau,\ K\in \kappa \}$ and $\tau\kappa=\{TK\ |\ T\in \tau,\ K\in \kappa\}$ are also $R$-bounded and satisfies $R(\tau+\kappa)\leq R(\tau)+R(\kappa)$ and $R(\tau\kappa)\leq R(\tau)R(\kappa)$.

Since the $R$-boundedness is stronger than the uniform boundedness, we are able to define \textit{$R$-sectorial} operator and \textit{$R$-angle} $\phi_{L}^{R}$ for a sectorial operator $L$ by replacing the uniform bound (3.1) to the $R$-bound. When $X$ is a UMD-space, it is known that a sectorial operator $L$ with a bounded imaginary powers of power angle $\theta_{L}$, is $R$-sectorial for $\phi_{L}^{R}\leq \theta_{L}$ \cite{CP01}. When $X=Y=L^{p}(D)$ for $p\in (1,\infty)$, the condition of the $R$-boundedness is equivalent to the condition

\begin{align*}
\Bigg\|\Big(\sum_{j=1}^{N}|T_jx_j|^{2}\Big)^{1/2}\Bigg\|_{L^{p}(D)}
\leq C \Bigg\|\Big(\sum_{j=1}^{N}|x_j|^{2}\Big)^{1/2}\Bigg\|_{L^{p}(D)}.  \tag{B.2}
\end{align*}\\

We say that a function $m: \mathbb{R}\backslash \{0\}\to B(X,Y)$ is a  \textit{Fourier Multiplier} on $L^{q}(\mathbb{R}; X)$ if the operator 

\begin{align*}
Kf={\mathcal{F}}^{-1}m(\cdot){\mathcal{F}}f,\quad \textrm{for}\ f\in {\mathcal{S}}(\mathbb{R}; X),
\end{align*}\\
extends to a bounded operator from $L^{q}(\mathbb{R}; X)$ to $L^{q}(\mathbb{R}; Y)$. It is known that for UMD-spaces $X$ and $Y$, a function $m\in C^{1}(\mathbb{R}\backslash \{0\}; B(X,Y))$ is a Fourier Multiplier on $L^{q}(\mathbb{R}; X)$ for all $q\in (1,\infty)$ if $m(\xi)$ and $\xi m'(\xi)$ are $R$-bounded for $\xi \in \mathbb{R}\backslash \{0\}$ \cite[3.4 Theorem]{Weis01}.

\vspace{15pt}

We apply a multiplier theorem on $L^{q}(\mathbb{R}; L^{p}(D))$ and estimate $u_{h}$ given by the formula (2.11). We use the boundedness of the pure imaginary powers of $B_{1}$

\vspace{15pt}

\begin{prop}
The estimate 

\begin{align*}
|\lambda|||u_{h}||_{L^{p}(\Pi)}\leq C ||f_h||_{L^{p}(\Pi)}   \tag{B.3}
\end{align*}\\
holds for $u_h$ given by the formula (2.11) for $f\in C^{\infty}_{c}(\Pi)$.
\end{prop}

\vspace{5pt}

\begin{proof}
We show that the function $m_1(\xi)=\lambda (\lambda+\xi^{2}+B_1)^{-1}$ is a Fourier multiplier on $L^{q}(\mathbb{R}; L^{p}(D))$ for all $q\in (1,\infty)$. Then, the estimate (B.2) follows from (2.11) by taking $q=p$. 

It suffices to show that $m_{1}(\xi)$ and $\xi m_{1}'(\xi)$ are $R$-bounded. Since the operator $B_1$ on $L^{p}(D)$ admits a bounded imaginary power of power angle zero by \cite{Seeley71}, \cite{Duong90}, it is an $R$-sectorial operator of $R$-angle zero \cite{CP01}. This means that for $\theta\in (\pi/2,\pi)$ there exists a constant $C$ such that 

\begin{align*}
R(\{\mu (\mu+B_1)^{-1}   \ |\ \mu \in \Sigma_{\theta}\})\leq C. \tag{B.4}
\end{align*}\\
Since $|\lambda|/|\lambda+\xi^{2}|\leq 1/\sin{\theta}$ for $\lambda\in \Sigma_{\theta}$ as in the proof of Proposition 2.5, it follows from (B.2) that 

\begin{align*}
R(\{\lambda(\lambda+\xi^{2})^{-1}\ |\ \xi\in \mathbb{R}\backslash \{0\}\ \})\leq \frac{C}{\sin{\theta}},\quad \lambda\in \Sigma_{\theta}.
\end{align*}\\
Hence we have

\begin{align*}
R(\{m_1(\xi)\})&=R(\{\lambda(\lambda+\xi^{2}+B_1)^{-1}\})\\
&\leq R(\{\lambda(\lambda+\xi^{2})^{-1}\})R(\{(\lambda+\xi^{2})(\lambda+\xi^{2}+B_1)^{-1}\})\\
&\leq \frac{C}{\sin{\theta}},\quad \lambda\in \Sigma_{\theta}.
\end{align*}\\
Since the resolvent is holomorphic, we are able to estimate an $R$-bound of $\xi m_1'(\xi)$ by using (B.4) as we did in the proof of Proposition 2.5. Thus the function $m_1$ is a Fourier multiplier on $L^{q}(\mathbb{R}; L^{p}(D))$ for all $q\in (1,\infty)$.
\end{proof}

\vspace{15pt}

We next estimate $u^{z}$. We set the domain of the operator $B_2$ by $D(B_2)=\{w\in W^{2,p}(D)\ |\ \partial_n w=0\ \textrm{on}\ \partial D \}$. Since the kernel of the operator $B_2$ is not empty on $L^{p}(D)$, it is not a sectorial operator in the sense of Section 3. We thus restrict the operator  to a space of average-zero functions $L^{p}_{0}(D)=\{h\in L^{p}(D)\ |\ \int_{D}h\dd x=0 \}$ by setting  

\begin{align*}
&\tilde{B}_2 w=-\Delta w\quad w\in D(\tilde{B}_2),\\
&D(\tilde{B}_2)=D(B_2)\cap L^{p}_{0}.
\end{align*}\\
Since the average of $\tilde{B}_2w$ in $D$ vanishes by the Neumann boundary condition, the operator $\tilde{B}_{2}$ is an invertible sectorial operator acting on $L^{p}_{0}$. Moreover, the operator $\tilde{B}_{2}$ admits a bounded imaginary power of power angle zero \cite[Theorem 2]{ST98}.

We show the estimate (B.1) for $u^{z}$ by applying a multiplier theorem for a resolvent of $\tilde{B}_2$. We consider functions $f_1\in C^{\infty}(\overline{\Pi})$ satisfying

\begin{equation*}
\begin{aligned}
\textrm{spt}\ f_1\ \textrm{is compact in}\ \overline{\Pi},\\
\int_{D}f_1(x_h,x_3)\dd x_h=0,\quad x_3\in \mathbb{R}.
\end{aligned}
\tag{B.5}
\end{equation*}\\
We see that ${\mathcal{F}}f_1$ is average-zero in $D$ and belongs to $L^{p}_{0}$ for $\xi\in \mathbb{R}$. Hence we use the resolvent of $\tilde{B}_2$ and set

\begin{align*}
u_1={\mathcal{F}}^{-1}(\lambda+\xi^{2}+\tilde{B}_2)^{-1}{\mathcal{F}}f_1.  \tag{B.6}
\end{align*}\\
The function $u_1$ is a solution of (2.3) for $f_1$ and its average in $D$  vanishes for each $x_3\in \mathbb{R}$.

\vspace{15pt}

\begin{prop}
There exists a constant $C$ such that 

\begin{align*}
|\lambda|||u_1||_{L^{p}(\Pi)}\leq C||f_1||_{L^{p}(\Pi)}  \tag{B.7}
\end{align*}\\
for $\lambda\in \Sigma_{\theta}$ and $f_1\in C^{\infty}(\overline{\Pi})$ satisfying (B.5).
\end{prop}

\vspace{5pt}

\begin{proof}
The assertion follows from a multiplier theorem as in the proof of Proposition B.2. 
\end{proof}

\vspace{15pt}

We subtract from $f^{z}$ the average of $f^{z}$ in $D$ and apply Proposition B.3.

\vspace{15pt}

\begin{prop}
The estimate 

\begin{align*}
|\lambda|||u^{z}||_{L^{p}(\Pi)}\leq C||f^{z}||_{L^{p}(\Pi)}  \tag{B.8}
\end{align*}\\
holds for $u^{z}$ given by the formula (2.11) for $f\in C^{\infty}_{c}(\Pi)$.
\end{prop}

\vspace{5pt}

\begin{proof}
We set the functions $f_1$ and $f_2$ by 

\begin{align*}
f_1(x_h,x_3)&=f^{z}-f_2,\\
f_2(x_3)&=\int_{D}f^{z}(x_h,x_3)\dd x_h.
\end{align*}\\
Since $f_1\in C^{\infty}(\overline{\Pi})$ satisfies (B.5), the function $u_1$ defined by (B.6) satisfies the Neumann problem (2.3) for $f_1$ and the estimate (B.7) holds by Proposition B.3. We set 

\begin{align*}
u_2={\mathcal{F}}^{-1}(\lambda+\xi^{2})^{-1}{\mathcal{F}}f_2.
\end{align*}\\
Then, $u_2$ satisfies $\lambda u_2-\partial_{x_3}^{2}u_2=f_2$. Since $m(\xi)=\lambda(\lambda+\xi^{2})^{-1}$ satisfies $|m(\xi)|\leq |\sin\theta|^{-1}$ and $|\xi m'(\xi)|\leq 2|\sin\theta|^{-2}$, the classical Mihlin multiplier theorem \cite[6.1.6 Theorem]{BL} implies that  

\begin{align*}
|\lambda|||u_2||_{L^{p}(\mathbb{R})}\leq C||f_2||_{L^{p}(\mathbb{R})}\quad \lambda\in \Sigma_{\theta}.
\end{align*}\\
Hence $u_2$ satisfies $|\lambda| ||u_2||_{L^{p}(\Pi)}\leq C||f_{2}||_{L^{p}(\Pi)}$. Since $u^{z}$ agrees with $u_1+u_2$, we obtain (B.8).
\end{proof}

\vspace{5pt}

\begin{proof}[Proof of Lemma B.1]
The estimate (B.1) holds for $u$ given by the formula for $f\in C^{\infty}_{c}(\Pi)$ by (B.3) and (B.8). For general $f\in L^{p}(\Pi)$, we take a sequence $\{f_m\}\subset C^{\infty}_{c}(\Pi)$ such that $f_m\to f$ in $L^{p}(\Pi)$ and  obtain the desired estimate. 
\end{proof}

\vspace{15pt}

\section*{Acknowledgements}
The author is partially supported by JSPS through the Grant-in-aid for Young Scientist (B) 17K14217, Scientific Research (B) 17H02853 and Osaka City University Strategic Research Grant 2018 for young researchers.

\vspace{15pt}

\bibliographystyle{plain}

\bibliography{ref}

\begin{thebibliography}{10}

\bibitem{A7}
K.~Abe.
\newblock {\em Vanishing viscosity limits for axisymmetric flows with
  boundary}.
\newblock arXiv:1806.04811.

\bibitem{Ad}
R.~A. Adams and J.~J.~F. Fournier.
\newblock {\em Sobolev spaces}, volume 140 of {\em Pure and Applied Mathematics
  (Amsterdam)}.
\newblock Elsevier/Academic Press, Amsterdam, second edition, 2003.

\bibitem{ADN}
S.~Agmon, A.~Douglis, and L.~Nirenberg.
\newblock Estimates near the boundary for solutions of elliptic partial
  differential equations satisfying general boundary conditions. {I}.
\newblock {\em Comm. Pure Appl. Math.}, 12:623--727, (1959).

\bibitem{AKST}
T.~Akiyama, H.~Kasai, Y.~Shibata, and M.~Tsutsumi.
\newblock On a resolvent estimate of a system of {L}aplace operators with
  perfect wall condition.
\newblock {\em Funkcial. Ekvac.}, 47:361--394, (2004).

\bibitem{Bardos72}
C.~Bardos.
\newblock Existence et unicit\'e de la solution de l'\'equation d'{E}uler en
  dimension deux.
\newblock {\em J. Math. Anal. Appl.}, 40:769--790, (1972).

\bibitem{BL}
J.~Bergh and J.~L{\"o}fstr{\"o}m.
\newblock {\em Interpolation spaces. {A}n introduction}.
\newblock Springer-Verlag, Berlin-New York, 1976.

\bibitem{BB74}
J.~P. Bourguignon and H.~Brezis.
\newblock Remarks on the {E}uler equation.
\newblock {\em J. Functional Analysis}, 15:341--363, (1974).

\bibitem{Catta}
L.~Cattabriga.
\newblock Su un problema al contorno relativo al sistema di equazioni di
  {S}tokes.
\newblock {\em Rend. Sem. Mat. Univ. Padova}, 31:308--340, (1961).

\bibitem{CP01}
P.~Cl\'ement and J.~Pr\"uss.
\newblock An operator-valued transference principle and maximal regularity on
  vector-valued {$L_p$}-spaces.
\newblock In {\em Evolution equations and their applications in physical and
  life sciences ({B}ad {H}errenalb, 1998)}, pages 67--87. Dekker, New York,
  2001.

\bibitem{DeHP}
R.~Denk, M.~Hieber, and J.~Pr{\"u}ss.
\newblock {$R$}-boundedness, {F}ourier multipliers and problems of elliptic and
  parabolic type.
\newblock {\em Mem. Amer. Math. Soc.}, 166(788):viii+114, (2003).

\bibitem{Duong90}
X.~T. Duong.
\newblock {$H_\infty$} functional calculus of elliptic operators with
  {$C^\infty$} coefficients on {$L^p$} spaces of smooth domains.
\newblock {\em J. Austral. Math. Soc. Ser. A}, 48:113--123, (1990).

\bibitem{EM70}
D.~G. Ebin and J.~Marsden.
\newblock Groups of diffeomorphisms and the motion of an incompressible fluid.
\newblock {\em Ann. of Math. (2)}, 92:102--163, (1970).

\bibitem{E}
L.~C. Evans.
\newblock {\em Partial differential equations}, volume~19 of {\em Graduate
  Studies in Mathematics}.
\newblock American Mathematical Society, Providence, RI, second edition, 2010.

\bibitem{FLR77}
E.~B. Fabes, J.~E. Lewis, and N.~M. Rivi\`ere.
\newblock Boundary value problems for the {N}avier-{S}tokes equations.
\newblock {\em Amer. J. Math.}, 99:626--668, (1977).

\bibitem{FR07a}
R.~Farwig and M.-H. Ri.
\newblock An {$L^q(L^2)$}-theory of the generalized {S}tokes resolvent system
  in infinite cylinders.
\newblock {\em Studia Math.}, 178(3):197--216, (2007).

\bibitem{FR07c}
R.~Farwig and M.-H. Ri.
\newblock The resolvent problem and {$H^\infty$}-calculus of the {S}tokes
  operator in unbounded cylinders with several exits to infinity.
\newblock {\em J. Evol. Equ.}, 7:497--528, (2007).

\bibitem{FR07b}
R.~Farwig and M.-H. Ri.
\newblock Stokes resolvent systems in an infinite cylinder.
\newblock {\em Math. Nachr.}, 280:1061--1082, (2007).

\bibitem{FR08}
R.~Farwig and M.-H. Ri.
\newblock Resolvent estimates and maximal regularity in weighted {$L^q$}-spaces
  of the {S}tokes operator in an infinite cylinder.
\newblock {\em J. Math. Fluid Mech.}, 10:352--387, (2008).

\bibitem{Gal}
G.~P. Galdi.
\newblock {\em An introduction to the mathematical theory of the
  {N}avier-{S}tokes equations}.
\newblock Springer Monographs in Mathematics. Springer, New York, second
  edition, 2011.

\bibitem{GHT}
M.~Geissert, H.~Heck, and C.~Trunk.
\newblock {$H^\infty$}-calculus for a system of {L}aplace operators with mixed
  order boundary conditions.
\newblock {\em Discrete Contin. Dyn. Syst. Ser. S}, 6:1259--1275, (2013).

\bibitem{GM}
Y.~Giga and T.~Miyakawa.
\newblock Solutions in {$L_r$} of the {N}avier-{S}tokes initial value problem.
\newblock {\em Arch. Rational Mech. Anal.}, 89:267--281, (1985).

\bibitem{Kato72}
T.~Kato.
\newblock Nonstationary flows of viscous and ideal fluids in {${\bf R}^{3}$}.
\newblock {\em J. Functional Analysis}, 9:296--305, (1972).

\bibitem{KatoLai}
T.~Kato and C.~Y. Lai.
\newblock Nonlinear evolution equations and the {E}uler flow.
\newblock {\em J. Funct. Anal.}, 56:15--28, (1984).

\bibitem{LM}
J.-L. Lions and E.~Magenes.
\newblock {\em Non-homogeneous boundary value problems and applications.}
\newblock Springer-Verlag, New York-Heidelberg, 1972.

\bibitem{Lunardi}
A.~Lunardi.
\newblock {\em Analytic semigroups and optimal regularity in parabolic
  problems}.
\newblock Progress in Nonlinear Differential Equations and their Applications,
  16. Birkh\"auser Verlag, Basel, 1995.

\bibitem{McIntosh86}
A.~McIntosh.
\newblock Operators which have an {$H_\infty$} functional calculus.
\newblock In {\em Miniconference on operator theory and partial differential
  equations ({N}orth {R}yde, 1986)}, volume~14 of {\em Proc. Centre Math. Anal.
  Austral. Nat. Univ.}, pages 210--231. Austral. Nat. Univ., Canberra, 1986.

\bibitem{Miyakawa80}
T.~Miyakawa.
\newblock The {$L^{p}$}\ approach to the {N}avier-{S}tokes equations with the
  {N}eumann boundary condition.
\newblock {\em Hiroshima Math. J.}, 10:517--537, (1980).

\bibitem{Miyakawa81}
T.~Miyakawa.
\newblock On the initial value problem for the {N}avier-{S}tokes equations in
  {$L^{p}$}\ spaces.
\newblock {\em Hiroshima Math. J.}, 11:9--20, (1981).

\bibitem{MiyakawaYamada92}
T.~Miyakawa and M.~Yamada.
\newblock Planar {N}avier-{S}tokes flows in a bounded domain with measures as
  initial vorticities.
\newblock {\em Hiroshima Math. J.}, 22:401--420, (1992).

\bibitem{Muramatsu74}
T.~Muramatu.
\newblock On {B}esov spaces and {S}obolev spaces of generalized functions
  definded on a general region.
\newblock {\em Publ. Res. Inst. Math. Sci.}, 9:325--396, (1974).

\bibitem{Seeley71}
R.~Seeley.
\newblock Norms and domains of the complex powers {$A_{B}z$}.
\newblock {\em Amer. J. Math.}, 93:299--309, (1971).

\bibitem{Sohr}
H.~Sohr.
\newblock {\em The {N}avier-{S}tokes equations}.
\newblock Birkh\"auser Advanced Texts: Basler Lehrb\"ucher. Birkh\"auser
  Verlag, Basel, 2001.

\bibitem{ST98}
H.~Sohr and G.~Th\"ater.
\newblock Imaginary powers of second order differential operators and
  {$L^q$}-{H}elmholtz decomposition in the infinite cylinder.
\newblock {\em Math. Ann.}, 311:577--602, (1998).

\bibitem{Stein70}
E.~M. Stein.
\newblock {\em Singular integrals and differentiability properties of
  functions}.
\newblock Princeton Mathematical Series, No. 30. Princeton University Press,
  Princeton, N.J., 1970.

\bibitem{Swann71}
H.~S.~G. Swann.
\newblock The convergence with vanishing viscosity of nonstationary
  {N}avier-{S}tokes flow to ideal flow in {$R_{3}$}.
\newblock {\em Trans. Amer. Math. Soc.}, 157:373--397, (1971).

\bibitem{Tanabe79}
H.~Tanabe.
\newblock {\em Equations of evolution}, volume~6 of {\em Monographs and Studies
  in Mathematics}.
\newblock Pitman (Advanced Publishing Program), Boston, Mass.-London, 1979.

\bibitem{Te75}
R.~Temam.
\newblock On the {E}uler equations of incompressible perfect fluids.
\newblock {\em J. Functional Analysis}, 20:32--43, (1975).

\bibitem{VK81}
M.~I. Vishik and A.~I. Komech.
\newblock Individual and statistical solutions of a two-dimensional {E}uler
  system.
\newblock {\em Dokl. Akad. Nauk SSSR}, 261:780--785, (1981).

\bibitem{Weis01}
L.~Weis.
\newblock Operator-valued {F}ourier multiplier theorems and maximal
  {$L_p$}-regularity.
\newblock {\em Math. Ann.}, 319:735--758, (2001).

\end{thebibliography}

\end{document}